\documentclass[12pt]{amsart}
\usepackage{epic,latexsym,amsbsy,amsmath,amscd,amssymb}
\usepackage{amsthm}
\usepackage{mathtools}
\usepackage{subcaption}
\usepackage{graphicx}
\usepackage{tikz}
\usetikzlibrary{positioning,arrows}
\usetikzlibrary{decorations.markings}
\usetikzlibrary{decorations.pathmorphing}
\usetikzlibrary{arrows.meta}
\usetikzlibrary{calc}
\tikzstyle{every node}=[font=\small]
\usepackage[font=small]{caption}
\usepackage{geometry, adjustbox}

\DeclareFontFamily{U}{mathx}{\hyphenchar\font45}
\DeclareFontShape{U}{mathx}{m}{n}{<-> mathx10}{}
\DeclareSymbolFont{mathx}{U}{mathx}{m}{n}
\DeclareMathAccent{\widebar}{0}{mathx}{"73}

\usepackage{hyperref}

\newcommand\AUT{\operatorname{\sf AUT}}
\newcommand\Ccal{\mathcal C}
\newcommand\V{\mathbf{V}}
\newcommand\N{\mathbb N}
\newcommand\Z{\mathbb{Z}}
\newcommand\Si{\mathbf{\Sigma}}
\newcommand\wh{\widehat}

\newcommand\then{\!\Rightarrow\!}%{\!\!\implies\!\!}
\newcommand\thens{\!\impliess\!}
\newcommand\ray{\rho}
\newcommand\SAW{\textrm{SAW}}
\newcommand\ol{\overline}
\newcommand\Pb{\mathbf{P}}
\newcommand\diam{\operatorname{\rm diam}}
\newcommand\impliess{\buildrel * \over \Rightarrow}%{\buildrel * \over \implies}
\newcommand\AND{\quad\text{and}\quad}
\newcommand\0{\kern 0.067em}

\begin{document}%$\,$ \vspace{-1truecm}

\newtheorem{thm}{Theorem}[section]
\newtheorem{pro}[thm]{Proposition}%[section]
\newtheorem{lem}[thm]{Lemma}%[section] 
\newtheorem{cor}[thm]{Corollary}%[section]

\theoremstyle{definition}
\newtheorem{con}[thm]{Convention}%[section]
\newtheorem{dfn}[thm]{Definition}%[section]
\newtheorem{exa}[thm]{Example}%[section]
\newtheorem{exs}[thm]{Examples}%[section]
\newtheorem{rmk}[thm]{Remark}%[section]

\title{The language of self-avoiding walks}
\author{\bf Christian LINDORFER and Wolfgang WOESS}
\address{\parbox{.8\linewidth}{Institut f\"ur Diskrete Mathematik,\\ 
Technische Universit\"at Graz,\\
Steyrergasse 30, A-8010 Graz, Austria\\ $\,$}}

\email{lindorfer@math.tugraz.at, woess@tugraz.at}

\date{\today} 
%\thanks{Partially supported by 
%Austrian Science Fund project FWF-P31889-N35}
   %Graphs and groups.
		  

\begin{abstract} Let $X=(V\!X,E\!X)$ be an infinite, locally finite, connected graph without loops or multiple edges. We consider the edges to be oriented, and $E\!X$ is equipped with an involution which inverts the orientation. Each oriented edge is  labelled by an element of a finite
alphabet $\Si$. The labelling is assumed to be deterministic: edges with the same initial
(resp. terminal) vertex have distinct labels. Furthermore it is assumed that the group of label-preserving automorphisms of $X$ acts quasi-transitively. For any vertex $o$ of $X$, consider the language of all words over $\Si$ which can be read along self-avoiding walks starting at $o$.
We characterize under which conditions on the graph structure this language is regular or context-free. This is the case if and only if the graph has more than one end, and the size of all ends is $1$, or at most $2$, respectively.
\end{abstract}

\maketitle

\markboth{{\sf C. Lindorfer and W. Woess}}
{{\sf The language of self-avoiding walks}}
\baselineskip 15pt

%%%%%%%%%%%%%%%%%%%%%%%%%%%%%%%%%%%

\section{Introduction}\label{sec:intro}

Let $X=(V\!X,E\!X)$ be a locally finite, connected, infinite graph without loops or multiple edges. 
We think of the edges as being directed, so that each  $e \in E\!X$ has an initial vertex $e^-$ and a terminal vertex $e^+$;
there is an involution $e \mapsto \check e$ which exchanges those endpoints. In drawing a
graph, every pair $(e,\check e)$ will usually be represented (and thought of) as one 
undirected edge. We shall always consider \emph{quasi-transitive} graphs, i.e., the automorphism
group of $X$ is supposed to act with finitely many orbits on $X$.

A \emph{walk} in $X$ is a subgraph spanned by a 
sequence $\pi = [e_1\,,e_2\,,\dots, e_n]$
with $e_k^- = e_{k-1}^+\,.$ Its length is $|\pi|=n$, and its initial and terminal 
vertices are $e_1^-$ and $e_n^+\,$, respectively. This comprises the empty walk 
starting (and ending) at a vertex, which has length $0$. 
A \emph{self-avoiding} walk (SAW) is one which
visits no vertex twice.  

We choose a root vertex $o$ and consider the number $c_n(o)$ of all SAWs of length $n$
starting at $o$. {\sc Hammersley}~\cite{Ham} proved that in a quasi-transitive
graph, the limit
$$
\mu(X) = \lim_{n \to \infty} c_n(o)^{1/n}
$$
exists and is independent of the choice of the root $o$. It is the reciprocal of
the radius of convergence of the SAW-generating function
$$
F_{\SAW}(t) = F_{\SAW}(t|o) = \sum_{n=0}^{\infty} c_n(o)\,t^n.
$$
The number $\mu(X)$ is called the \emph{connective constant} of $X$. 
The study of self-avoiding walks, primarily in lattice graphs, was initiated by
the chemist and Nobel laureate {\sc Paul J. Flory}~\cite{Fl} as a tool for studying  
polymer growth. It is a combinatorial issue whose mathematical study has been advanced
to a large extent within Statistical Physics and Probability.
Good references are the monograph by {\sc Madras and Slade}~\cite{MaSl} and
the lecture notes by {\sc Bauerschmidt et al.}~\cite{BDGS}. The explicit computation
of $F_{\SAW}(t)$ or just $\mu(X)$ is a difficult task. A highlight is the rigorous
result of {\sc Duminil-Copin and Smirnov} ~\cite{DuSm}  for the hexagonal lattice,
while even for the standard square lattice, $\mu(X)$ is not known rigorously.

Here, we propose an apparently new approach, by connecting self-avoiding walks with the
theory of formal languages. Given a finite alphabet $\Si$, a language over $\Si$
is a subset $L$ of $\Si^*$, the free monoid over $\Si$ consisting of all
words with finite length whose letters come from $\Si$. Our setting is as follows.
We have a pair $(X,\ell)$, where $X$ is a graph as above, and $\ell: E\!X \to \Si$
is a \emph{labelling} which assigns to every oriented edge $e$ a label $\ell(e)$. 
Our assumptions are that the labelling is \emph{deterministic,} that is,
different edges with the same initial or terminal vertex have distinct labels, and that the group
$\AUT(X,\ell)$ of all graph automorphisms of $X$ which preserve $\ell$ acts
quasi-transitively (with finitely many orbits). The most significant class of
labelled graphs are the \emph{Cayley graphs} of finitely generated groups. 

The labelling extends to all walks:
for $\pi=[e_1\,,e_2\,,\dots, e_n]$, we set 
$$
\ell(\pi) = \ell(e_1)\ell(e_2)\dots\ell(e_n) \in \Si^*.
$$ 
For any set of walks $\Pi$ in the labelled graph $X$, the 
associated language is 
$$
L(\Pi)= \{ \ell(\pi) : \pi \in \Pi \}\,.
$$
Now let $\Pi_{\SAW} = \Pi_{\SAW,o}$ be the set of all self-avoiding walks
starting at $o$. The associated \emph{language of self-avoiding walks} is
$L_{\SAW} = L_{\SAW,o}(X) =  L\bigl(\Pi_{\SAW,o}\bigr)$.

In the Chomsky-hierarchy of formal languages, the first basic class
consists of the \emph{regular languages,}  which are those accepted
by a finite state automaton, or equivalently, generated by a 
right-linear grammar. The second class consists of the \emph{context-free}
languages, which are those accepted by a pushdown automaton, resp., generated
by a context free grammar. (Here, we shall use grammars.) Further details
will be given below, and an excellent reference is the monograph by
{\sc Harrison}~\cite{Ha}. 
Our main result is the following.

\begin{thm}\label{thm:main} Let $(X,\ell)$ be a deterministically labelled,
quasi-transitive graph, connected and locally finite. For any choice of the
root vertex, the following holds. 
\\[3pt]
\emph{(a)} $L_{\SAW}$ is regular if and only if $X$ has more than one end
and all ends have size $1$.
\\[3pt]
\emph{(b)} $L_{\SAW}$ is context-free if and only if $X$ has more than one end
and all ends have size at most $2$. In this case, $L_{\SAW}$ is unambiguous
context-free. 
\end{thm}

See below for the meaning of ``unambiguous''. A precise definition of an end of 
a graph and its size will also be given below; a reference close to the
spirit of the present paper is {\sc Thomassen and Woess}~\cite{ThWo}.

Relating walks in labelled graphs, in particular Cayley graphs, with 
formal language theory has an important history. Let $\Pi(o,o)$ be the 
set of all walks in $X$ starting and ending at $o$, possibly with several
self-intersections. Let $W = W(o) = L\bigl(\Pi(o,o)\bigr)$. For a Cayley
graph of a group, this language is called the \emph{word problem.}

{\sc Anisimov}~\cite{An} showed that the word problem is regular if and
only if the group is finite, and this extends to quasi-transitive labelled
graphs. In ground-breaking work, {\sc Muller and Schupp}~\cite{MS1} showed
that the word problem is context-free if and only if the group
is virtually free. In particular, regularity, resp. context-freeness of the
word problem are group invariants which do not depend on the specific
generating set. In subsequent work \cite{MS2}, context-free labelled graphs
were defined via structural properties (not necessarily quasi-transitive), 
and these are precisely the (deterministically) labelled graphs for which
$W(o)$ is context-free; see {\sc Ceccherini and Woess}~\cite{CeWo2}.
For further work on language-theoretic issues related with groups , see e.g. 
{\sc P\'elecq}~\cite{Pe}, as well as \cite{CeWo1}, \cite{Wcf2}, and for a new proof
of the main result of \cite{MS1} and related material, {\sc Dieckert and
Weiss}~\cite{DW}.

Applied to a Cayley graph of a group, Theorem \ref{thm:main} says that the group is virtually free if $L_{SAW}$ is context-free, but the latter property is not a group invariant.

The inspiration for the present work came from a note by 
{\sc Gilch and M\"uller}~\cite{GM}, who determined $F_{\SAW}(t)$
for free products of finite graphs -- an instance of the 
case where $L_{\SAW}$ is regular. Furthermore, in the computation by
{\sc Zeilberger}~\cite{Z} of $F_{\SAW}(t)$ for the bi-infinite ladder graph,
a context-free grammar is inherent although not mentioned or used
directly. 

This paper is organised as follows. In \S \ref{sec:basics}, we provide the 
necessary background on the end space of graphs, as well as on regular
and context-free languages. In \S \ref{sec:ladders}, we discuss
strips in locally finite graphs, that is, two-ended quasi-transitive subgraphs.
Their ends have the same finite size -- the size of the strip. 
In the quasi-transitive case, if there is an end of finite size $m$,  
there must be a strip of the same size. Furthermore, if there is a thick end
which is fixed by some non-elliptic automorphism (i.e., one which does not
fix a finite subset of $V\!X$), then there are strips of arbitrary size.
In \S \ref{sec:ends}, the Pumping Lemmas for regular, resp. context-free languages are then used 
to show the following. If $X$ is quasi-transitive and contains a strip of size 2, then
$L_{SAW}$ cannot be regular, and if it contains a strip of size 3, then
$L_{SAW}$ cannot be context-free.  

Thus, we are left with considering graphs whose ends have size at most 2.
In \S \ref{sec:cf} we first consider the case when all ends have size 1.
Then the cut-vertex tree decomposition of $X$ has finite blocks, and we derive
that $L_{\SAW}$ is regular. If all ends have size $2$, 
then we use the 3-block tree decomposition of $X$ of {\sc Droms, Servatius and
Servatius}~\cite{DSS} to construct an unambiguous context-free grammar for 
$L_{\SAW}$. (Alternatively, one might use the vertex cuts of {\sc Dunwoody
and Kr\"on}~\cite{DK}.) To conclude, if $X$ contains ends of both sizes $2$ and $1$, 
one can combine  $2$-connectedness of the (possibly infinite) blocks of the cut-vertex
tree decomposition with the method of the preceding case (ends of size 1) to get
context-freeness.

In the final \S \ref{sec:final}, we start with a discussion of implications and future work, recall some context-free examples, including one of {\sc Lindorfer}~\cite{Li}, and provide an additional more detailed example, where the SAW-generating function is algebraic over $\mathbb{Q}$, but not rational.

\section{Basic background}\label{sec:basics}

\subsection{Ends and automorphisms of graphs\nopunct}
$\,$

\smallskip

Recall that for our language-theoretic approach, it is convenient to consider the
edge set $E\!X$ of our (locally finite, connected, infinite) graph $X$ as being directed 
and with an involution which inverts the 
orientation. When speaking about ends, it is sufficient to identify each pair of oppositely
oriented edges with the same endpoints with one non-oriented edge. 
We shall frequently switch back and forward between these two viewpoints. 
The standard graph distance in a connected graph $X$ is denoted $d_X(\cdot,\cdot)$.
Walks will also be written in terms of vertices instead of edges,
so that $\pi=[x_0\,,x_1\,,\dots, x_n]$ with $x_i \in V\!X$ is the same as 
$[e_1\,,\dots, e_n]$, where $e_i = [x_{i-1}\,,x_i] \in E\!X$. ``Path'' is a synonym for 
``self-avoiding walk''.

The space
of ends of a connected graph was introduced in Graph Theory by 
{\sc Halin}~\cite{Ha1}, and -- without graph terminology -- earlier by
{\sc Freudenthal}~\cite{Fr}.

If $K \subset V\!X$, then $X\setminus K$ is the subgraph obtained from $X$ by removing $K$ 
and all edges incident to vertices in $K$. If removing $K$ disconnects $X$ then
$K$ is called a \emph{separating set}. In this case, if $K = \{ x\}$ then $x$
is called a cut-vertex. A \emph{ray} is a subgraph of $X$ spanned by a sequence
$\ray=[x_0\,,x_1\,,x_2\,,\dots]$  of distinct vertices such that 
$[x_{i-1}\,,x_i] \in E\!X$ for all $i$. 
A \emph{double ray} is defined analogously.
Two rays are called \emph{equivalent,} if for any finite $K \subset V\!X$, 
both end up in the same component of $X \setminus K$, i.e., all but finitely many
of their  vertices are contained in that component. 
An \emph{end} is an equivalence class of rays. If $\omega$ is an end and
$K \subset X$ is finite, then we write $C(K,\omega)=C_X(K,\omega)$ 
for the unique component of $X \setminus K$ in which all representing rays of $\omega$
end up, and we say that $\omega$ belongs to that component. 

A \emph{defining sequence} for an end $\omega$ consists of a sequence
$(K_n)_{n \ge 0}$ of finite subsets of $V\!X$ such that
$C(K_{n-1}\,,\omega) \supset K_n \cup C(K_n\,,\omega)$ for each $n$.
If there is a finite number $m$ such that $|K_n|=m$ for all $n$,
then the end is called \emph{thin,} and the minimal such $m = m(\omega)$ 
is called
the \emph{size} of $\omega$. Otherwise, the end is called \emph{thick}
with size $+\infty\,$. It is a consequence of Menger's theorem that
an end of size $m \le \infty$ contains (as an equivalence class) 
$m$ disjoint rays, and $m$ is maximal with respect to this property.

The space $\Omega$ of all ends is the boundary of a compactification of
$X$: the topology on $\wh X = X \cup \Omega$ is discrete on $X$.
In the above notation, let $\wh C(K,\omega)$ be the union of $C(K,\omega)$ with the set of all 
ends belonging to that component. Then for any defining sequence $(K_n)_{n \ge 0}$
of $\omega$, the family $\wh C(K_n,\omega)$ ($n \ge 0$) is a neighbourhood
base of $\omega$. The vertex set is open and dense in $\wh X$.

The \emph{automorphism group} $\AUT(X)$ consists of all permutations of $V\!X$
which preserve neighbourhood. A subgroup $\Gamma \le \AUT(X)$ is said
to act quasi-transitively if there are only finitely many orbits 
$\Gamma x = \{\gamma x : \gamma \in \Gamma \}$ ($x \in V\!X$), 
and transitively, if  $\Gamma x = V\!X$ for some (every) $x \in V\!X$.
In this case, the graph itself is called (quasi-)transitive.
It is well-known that an infinite, locally finite, connected graph
which is quasi-transitive has one, two or infinitely many ends, see
\cite{Fr}. If it has one end, then this end is thick. If it has two 
ends, then both are thin and have the same size (this situation will be
important in \S \ref{sec:ladders}). If it has infinitely
many ends, then there are thin ends. See e.g. \cite{Ha2}.

Now we return to the setting where the oriented edges of $X$ are labelled
by a finite alphabet $\Si$. We always assume that the labelling 
$E\!X \ni e \mapsto \ell(e) \in \Si$ is deterministic. The most typical example
arises when $\Gamma$ is a finitely generated group and $S$ is a finite,
symmetric set of generators. The \emph{Cayley graph} $X(\Gamma,S)$ of $\Gamma$ with respect 
to $S$ has 
vertex set $\Gamma$. We choose $\Si = S$, and 
for each $x \in \Gamma$ and $s \in S$ there is an oriented edge
from $x$ to $xs$ with label $s$. The group $\Gamma$ acts transitively 
by multiplication from the left. More generally, we have the following.

\begin{lem}\label{lem:fingen}
The group $\Gamma = \AUT(X,\ell)$ of all graph automorphisms of the connected graph
$X$ which preserve the edge labels acts fixed-point freely: if $\gamma \in \Gamma$
and $\gamma x = x$ for some $x \in V\!X$ then $\gamma = 1_{\Gamma}\,$, the unit element
of $\Gamma$.

In particular, if $\Gamma$ acts quasi-transitively then it is finitely generated. 
\end{lem}

\begin{proof} Suppose $\gamma x = x$. Since the labelling is deterministic,
 $\gamma y = y$ for all neighbours of $x$. By connectedness of $X$, we must have
$\gamma = 1_G\,$. If $\Gamma$ acts quasi-transitively then let $D$ be the diameter
of the (finite, connected) factor graph $\Gamma \backslash X$. Given $x, y \in V\!X$,
we have $d_X(y,\Gamma x) \le D$. Now consider the new graph $X^{2D+1}$
with the same vertex set $V\!X$, where two vertices $x,y$ are connected by a
non-oriented edge whenever $1 \le d_X(x,y) \le 2D+1$. In $X^{2D+1}$,
each orbit $\Gamma x$ induces a connected, locally finite subgraph on which 
$\Gamma$ acts transitively and fixed-point freely. Therefore that subgraph is 
a Cayley graph of $\Gamma$, and $\Gamma$ is
finitely generated, see e.g. the old note by {\sc Sabidussi}~\cite{Sa}.
\end{proof}

Even if our main interest is in Cayley graphs, we need the quasi-transitive
case throughout our proofs.

At this point, we also recall the notion of a \emph{quasi-isometry} between two
metric spaces $(X, d_X)$ and $(Y, d_Y)$. This is a mapping $\varphi: X \to Y$
such that there are constants $A > 0$, $B, B' \ge 0$ such that for all
$x_1\,,x_2 \in X$ and $y \in Y$,
\begin{equation}\label{eq:qi}
A^{-1}d_X(x_1\,,x_2) - B \le d_Y(\varphi x_1\,,\varphi x_2) \le  A\,d_X(x_1\,,x_2) + B
\AND d_Y(y, \varphi X) \le B'.
\end{equation}
Every quasi-isometry $\varphi$ has a \emph{quasi-inverse} $\psi: Y \to X$,
a quasi-isometry such that $\psi\, \varphi$ and  $\varphi\, \psi$
are at bounded distance from the respective identity mappings.

By a quasi-isometry between two connected graphs, we mean a quasi-isometry between
the vertex sets, equipped with the respective standard graph metrics.

\begin{rmk}\label{rmk:qi}
For any two Cayley graphs of the same finitely generated group $\Gamma$ with respect to 
two different finite, symmetric sets of generators, the identity mapping is a 
quasi-isometry with $B=B'=0$ in \eqref{eq:qi}, i.e., the mapping is \emph{bi-Lipschitz}. 
In the situation of Lemma \ref{lem:fingen} and its proof, the identity mapping on
any orbit $\Gamma x$ is a quasi-isometry from the Cayley graph of $\Gamma$ 
(given as the subgraph of $X^{2D + 1}$ induced by that orbit) to the original graph $X$.
Indeed, this is a quasi-surjective bi-Lipschitz embedding, i.e., $B=0$ in~\eqref{eq:qi}.
\end{rmk}

In the following Lemma, the subscripts $X$ and $Y$ refer to the respective graphs,
their metrics, and so on.
%The following is well-known and a nice exercise, see e.g. \cite{???}.

\begin{lem}\label{lem:qi}
Let $X$ and $Y$ be two connected graphs with bounded vertex degrees.
If $\varphi: X \to Y$ is a quasi-isometry, then it extends to a continuous
mapping $\wh X \to \wh Y$ which restricts to a homeomorphism between
the spaces of ends of $X$ and $Y$. 

There is an increasing function $\theta=\theta_{\varphi}: \N \to \N$ with the following property:
if $\omega \in \Omega_X$ and $K \subset X$ with $\diam_X(K) = k$ then there is 
$\,\ol{\!K} \subset Y$ with 
$$
\varphi K \subset \,\ol{\!K}\,,\quad \diam_Y(\,\ol{\!K}) \le \theta(k), \AND
\varphi C_X(K,\omega) \supset C_Y(\,\ol{\!K},\varphi\omega) \cap \varphi X.
$$
In particular, if $\omega$ has a defining sequence $(K_n)$ with $\diam_X(K_n) \le k < \infty$
then $\varphi\omega$ has a defining sequence $\,\ol{\!K}_n$ with $\diam_Y(\,\ol{\!K}_n) 
\le \theta(k)\,$, and if $\omega$ is a thick end, then so is $\varphi \omega$.
\end{lem}

This follows with some small additional effort from \cite[Lem. 21.3 and 21.4]{Wbook} and
their proofs. Using \cite[Thm. 4.4]{ThWo} and its proof for the case where both $X$ and $Y$ are quasi-transitive, we obtain the following.

\begin{cor}\label{cor:qi}
Let $X$ and $Y$ be connected, quasi-transitive graphs and $\varphi: X \to Y$ be a quasi-isometry. Then there is an increasing function $\eta=\eta_{\varphi}: \N \to \N$ such that any end $\omega \in \Omega_X$ of size $k$ maps to an end $\varphi \omega \in \Omega_Y$ of size at most $\eta(k)$.
\end{cor}
 
\subsection{Regular and context-free languages\nopunct}
$\,$

\smallskip

As mentioned above, \cite{Ha} is an optimal source on context-free languages.
We recall from the introduction that for any alphabet (set) $\Si$,
we denote by
$$
\Si^* = \{ w=a_1a_2\cdots a_n : n \ge 0,\; a_i \in \Si \}
$$ 
the set of all words $w$ over $\Si$. Here, $|w|=n$ is the \emph{length}
of $w$, and when $n=0$, this is the \emph{empty word} $\epsilon$.

A \emph{context-free grammar} is a quadruple $\Ccal = (\V,\Si,\Pb,S)$, 
where  $\V$ is the finite set of \emph{variables} (with 
$\V\cap \Si = \emptyset$), the variable $S$ is the \emph{start symbol,} 
and $\Pb \subset \V \times (\V \cup \Si)^*$ is a 
finite set of \emph{production rules.} We write $A \vdash u$ 
if $(A,u) \in \Pb$.  For $v, w \in (\V \cup \Si)^*$, a \emph{rightmost 
derivation step} has the form $v \then w$, where $v = v_1Av_2$ and 
$w=v_1uv_2$ with $u, v_1 \in (\V \cup \Si)^*$, $v_2 \in \Si^*$ and $A \vdash u$. 
A \emph{rightmost derivation} is a finite sequence
 $v=w_0\,, w_1\,, \dots,w_k=w \in (\V \cup \Si)^*$ such 
that $w_{i-1} \then w_i\,$; we then write $v \thens w$. 
Each $A \in \V$ generates the language  
$L_A = \{ w \in \Si^* : A \thens w \}$. The \emph{language generated by} 
$\Ccal$ is $L(\Ccal) =L_S$. The grammar is called \emph{unambiguous,} if
every $w \in  L(\Ccal)$ has a unique rightmost derivation.

A grammar and the language which it generates are called \emph{linear,}
if each production is of the form 
$$
A \vdash uBv \quad\text{or}\quad A \vdash u,\quad \text{where}
\quad A, B \in \V,\; u,v \in \Si^*.
$$
and if in that situation one always has $v = \epsilon$, then
the grammar and the language are called \emph{right-linear} or \emph{regular}.
In this case, the language is accepted by a (deterministic)
finite state automaton, see \cite{Ha}.

Typical tools to show that a language is \emph{not} regular, resp. context-free, 
are the well known Pumping Lemmas. 
\begin{lem}[Pumping Lemma for regular languages] \label{lem:pumplemreg}
Let $L$ be a regular language. Then there is a pumping length $p>0$ such 
that every $z \in L$ with $|z| \geq p$ can be written as $z=u\0 v\0\tilde u$, where 
$|v\0 \tilde u| \leq p$, $|v| \geq 1$ and $u\0 v^n\tilde u \in L$ for all $n \geq 0$.
\end{lem}
\begin{lem}[Pumping Lemma for context-free languages] \label{lem:pumplemcf}
Let $L$ be a context-free language. Then there is a pumping length 
$p>0$ such that every $z \in L$ with $|z| \geq p$ can be written as 
$z=u\0v\0w\0\tilde v\0 \tilde u$, where $|v\0w\0\tilde v| \leq p$, $|v\0\tilde v| \geq 1$ and 
$u\0v^n\0w\0\tilde v^n\0\tilde u \in L$  for all $n \geq 0$.
\end{lem} 

\section{Strips in locally finite graphs}\label{sec:ladders}

The action of the automorphism group of a locally finite, connected graph extends 
in an obvious way to the space of ends. The automorphisms can be 
classified into 3 types:
\begin{itemize}
 \item $\gamma \in \AUT(X)$ is \emph{elliptic,} if it fixes a finite subset of $V\!X$,
 \item a non-elliptic $\gamma \in \AUT(X)$ is \emph{parabolic,} if it fixes a unique end, and
 \item it is \emph{hyperbolic,} if it fixes each of a unique pair of ends.  
\end{itemize}

While this terminology was not used by {\sc Halin}~\cite{Ha2}, he showed that for non-elliptic automorphism $\gamma$ and every $x\in V\!X$ the sequence $[x\,,\gamma x\,, \gamma^2x\,, \dots]$ uniquely defines an end of $X$ which is called the \emph{direction} of $\gamma$, denoted by $D(\gamma)$ and fixed by $\gamma$. The following theorem serves as one of the main pillars for our results.

\begin{thm}[{{\sc Halin}~\cite[Thm. 9]{Ha2}}] \label{thm:strip}
Let $\gamma$ be a non-elliptic automorphism acting on the locally finite connected graph $X$. Then the following holds:
\begin{enumerate}
 \item[(a)] $D(\gamma)$ and $D(\gamma^{-1})$ have the same size $m$.
\item[(b)] $D(\gamma) \neq D(\gamma^{-1})$ (i.e. $\gamma$ is hyperbolic) if and only if $m<\infty$.
\item[(c)] There are $m$ disjoint double rays $\rho_1\,, \rho_2\,, \dots$ which are invariant under $\gamma$.
\item[(d)] If $\gamma$ is hyperbolic then there are a set $K \subset V\!S$ with $|K|=m$ and an
integer $k$ such that $(\gamma^{kn}K)_{n \ge 0}$ and $(\gamma^{-kn}K)_{n \ge 0}$ are defining sequences for $D(\gamma)$ and $D(\gamma^{-1})$ respectively. Each $\rho_i$ $(i \le m)$  meets every $\gamma^{kn}K$ ($n \in \Z$) in precisely one vertex.
\end{enumerate}
\end{thm}

\begin{dfn}\label{def:strip} A locally finite, connected graph $S=(V\!S,ES)$ is called
 a \emph{strip} if it is quasi-transitive and has precisely two ends.
\end{dfn}

The structure of strips is well understood. We collect without proof those basic facts which will be needed below. More about strips can also be found in {\sc Jung and Watkins}~\cite{JW} and {\sc Imrich and Seifter}~\cite{IS1}, \cite{IS2}. 

For a strip $S$ there is some hyperbolic automorphism $\gamma$ fixing each of the two ends $\omega^+$ and $\omega^-$ of $S$. Thus by the previous theorem $\omega^+$ and $\omega^-$ have the same finite size $m$. Moreover it provides 
\begin{itemize}
\item a finite set $K \subset V\!S$ with $|K|=m$, together with
\item an automorphism $\tau \in \AUT(S)$, such that $(\tau^nK)_{n \ge 0}$ and $(\tau^{-n}K)_{n \ge 0}$ are defining sequences for $\omega^+$ and $\omega^-$, respectively and 
\item $m$ disjoint, $\tau$-invariant double rays, each of them intersecting every $\tau^n K$ in precisely one vertex.
\end{itemize}
By replacing $\tau$ with a suitable power $\tau^k$, we can assume that the subgraph $Y$ of $S$ spanned by $K \cup C(K,\omega^+) \setminus C(\tau K,\omega^+)$ is finite and connected. 

In this situation we call $S$ a $\tau$-strip of size $m$. We use the same terminology if $S$ is a %(not necessarily induced) 
subgraph of a bigger graph $X$, and $\tau \in \AUT(X)$ is an automorphism whose 
restriction to $S$ has the above properties.

The following lemma refines the well-known argument that in a quasi-transitive graph with more
than one end,  the directions of hyperbolic automorphisms are dense in the space of ends.

\begin{lem}\label{lem:size} Let $X$ be locally finite and connected and $\Gamma \le \AUT(X)$
act quasi-transitively.
If $X$ has a thin end of size $m$ then it contains a $\tau$-strip of size $m$ for some
$\tau \in \Gamma$.
\end{lem}

\begin{proof} 
Let the end $\omega$ of $X$ have size $m$, take $m$ disjoint rays in $\omega$ and let $(K_n)_{n \ge 0}$
be a defining sequence for $\omega$ such that $|K_n|=m$ for all $n$ and each $K_n$ meets each of the $m$ rays. Fix a finite set $R$ of representatives of the orbits given by the action of $\Gamma$ on $X$. For $n \geq 0$ write $C_n=C(K_n\,,\omega)$ and let $\gamma_n \in \Gamma$ such that $\gamma_n K_n$ contains an element of $R$. Then every $\gamma_n K_n$ is a tight $m$-vertex cut, i.e., a set of cardinality $m$ such that $X \setminus K_n$ has at least two components and every vertex of $K_n$ has at least one neighbour in each of them. Also $\gamma_n C_n$ is a component of $X \setminus  \gamma_n K_n$.

By \cite[Prop. 4.2]{ThWo} there are only finitely many tight $m$-vertex cuts containing some given vertex $r \in R$ and clearly every cut splits the graph in finitely many components, hence $\{(\gamma_n K_n\,, \gamma_n C_n) : n \geq 0\}$ is a finite set. Pick some $j>i \geq 0$ such that $(\gamma_i K_i\,, \gamma_i C_i)=(\gamma_j K_j\,, \gamma_j C_j)$ and let $\tau= \gamma_j^{-1} \gamma_i $. Then 
$$
C_i \supset \tau (K_i \cup C_i).
$$
We show that $\tau$ is hyperbolic and the direction $D(\tau)$ has size $m$.
Indeed, $(\tau^n K_i)_{n\ge 0}$ is a defining sequence for the end $D(\tau)$ belonging to all components $\tau^n C_i$, where $n \in \Z$. In particular $D(\tau)$ has size at most $m$. By Theorem \ref{thm:strip}, $\tau$ is hyperbolic. On the other hand there are $m$ disjoint paths from $K_i$ to $\tau K_i$. Their images under $\tau^n$ $(n \in \Z)$ build $m$ disjoint $\tau$-invariant double rays $\rho_1, \dots \rho_m$, implying that $D(\tau)$ has size $m$, as required. Add to those double rays a finite collection of finite paths connecting those double rays with each other and all their images under $\tau^n$, $n \in \Z$. After possibly replacing $\tau$ by a suitable power $\tau^k$, $(k\geq 1)$, we obtain a subgraph of $X$ which is a $\tau$-strip of size $m$.
\end{proof}

\begin{lem}\label{lem:parabolic}
 Let $X$ be locally finite and connected and $\Gamma \le \AUT(X)$ act quasi-transitively on $X$. If $\Gamma$ contains a parabolic element then for every $m \ge 1$, $X$ contains a $\tau$-strip of size at least $m$ for some suitable $\tau \in \Gamma$.
\end{lem}

\begin{proof} Suppose that $\gamma \in \Gamma$ is parabolic and $\omega$ is the unique
 end fixed by $\gamma$. By Theorem \ref{thm:strip}, there are countably many disjoint double rays
$\ray_n\,$, $n \in \N$, which are invariant under $\gamma$ and represent $\omega$. 
For $m \in \N$, we can find some connected, finite subgraph $K$ of $X$ which meets
each of $\ray_1\,,\dots, \ray_m\,$. Then the subgraph spanned by $\ray_1\,,\dots, \ray_m$ together with all the $\gamma^nK$, $n\in \Z$, is a strip and its two ends have size at least $m$. Application of Lemma $\ref{lem:size}$ concludes the proof.
\end{proof}

\begin{rmk}
Whenever a graph $X$ contains a $\tau$-strip of size $m$, it also contains a $\tau^3$-strip of any size in $l \in \{1,\dots,m-1\}$. This can be seen by deleting $m-l$ vertices from the set $K$ from the definition of $\tau$-strips and their images under $\tau^{3n}$, $(n\in \Z)$.
\end{rmk}

\section{Context-freeness and ends}\label{sec:ends}

In this section we prove the first half of Theorem \ref{thm:main}, namely

\begin{thm}\label{thm:main-1}
Let $X$ be a connected, locally finite, deterministically edge-labelled graph on which 
$\AUT(X,\ell)$ acts quasi-transitively.
\begin{enumerate} 
\item If $L_{SAW}$ is context-free, every end of $X$ has size at most 2.
\item If $L_{SAW}$ is regular, every end of $X$ has size 1.
\end{enumerate}
\end{thm}

The proof is based on the following two lemmas and two propositions.

%, where we always assume
%that the graph $X$ is as stated above in Theorem \ref{thm:main-1}.

\begin{lem}\label{lem:subgraph} Let $X'$ be a subgraph of $X$ which is invariant
 under a subgroup $\Gamma'$ of $\AUT(X,\ell)$ acting quasi-transitively on $X'$.
Suppose that $L_{\SAW,o}(X)$ is regular, resp. context-free. Then there is $o' \in V\!X'$
such that $L_{\SAW,o'}(X')$ is also regular, resp. context-free.
\end{lem}

\begin{proof} $X'$ is also a deterministically labelled graph. If we choose any $o' \in V\!X'$
 and $\Pi'$ is the set of \emph{all} walks in $X'$ starting at $o'$, then $L(\Pi')$ is
clearly a regular language. Namely, the factor graph $\Gamma' \backslash X'$ is a finite
state automaton for $L(\Pi')$, easily converted into a right-linear grammar; see \cite{Ha}.

Now there is some $o' \in V\!X'$ with $d(o,V\!X')=d(o,o')$, and there is a path $\pi_0$ in $X$
of that length from $o$ to $o'$. Let $v_0 = \ell(\pi_0)$, and let $\pi_0 \circ \Pi'$
be the set of all concatenated paths $\pi_0 \circ \pi'$, where $\pi' \in \Pi'$.
Thus, $L(\pi_0 \circ \Pi') = v_0\0 L(\Pi') = \{ v_0w : w \in  L(\Pi')\}$ is again regular. 

If $L_{\SAW,o}(X)$ is regular (resp. context-free), then by \cite{Ha} also 
$L_{\SAW,o}(X) \cap L(\pi_0 \circ \Pi')$ is regular (resp. context-free). Since 
$o'$ is the only vertex of $\pi_0$ which is contained in $V\!X'$, 
$$
L_{\SAW,o}(X) \cap L(\pi_0 \circ \Pi') = v_0\0 L_{\SAW,o'}(X').
$$
If we delete from the latter language the common prefix $v_0\,$, we also get
a regular (resp. context-free) language.  
\end{proof}

\begin{lem}\label{lem:torsion}
Let $X$ be a connected, infinite, locally finite and deterministically 
labelled graph and let $\Gamma \leq AUT(X,\ell)$ act quasi-transitively on $X$. 
Assume that $L_{SAW} = L_{\SAW,o}$ is context-free for a choice of $o \in V\!X$. 
Then $\Gamma$ contains a non-elliptic element.
\end{lem}

\begin{proof} 
Let $p$ be the pumping length of $L_{\SAW}$ given by Lemma \ref{lem:pumplemcf}.
Let $z \in L_{SAW}$ with $|z|\geq p$. Then it can be written as 
$z=uvw\tilde v \tilde u$ where $|vw\tilde v|\leq p$, $|v\tilde v|\geq 1$ 
and $uv^nw\tilde v^n \tilde u \in L_{SAW}$ for all $n \geq 0$. 
Now either $|v|>0$ or $|v|=0$ and $|\tilde v|>0$. Set $a=u$, $b=v$ in the first case and 
$a=uw$, $b=\tilde v$ in the second case (so that $a, b \in \Si^*$ and $|b| > 0$).
Let $x_0$ be the end-vertex of the path starting at $o$ and labelled by the word $a$.
Then, for every $n\geq 0$, we have the unique self-avoiding walk $\pi_n$ of length
$n|b|$ which starts at $x_0$ and has label $b^n$. Thus, $\pi_{n+1}$ is a self-avoiding
extension of $\pi_n$, and in the limit we obtain a ray $\ray = [x_0\,,x_1\,,\dots]$.
Using that $\Gamma$ acts quasi-transitively on $X$ there must be some $\tau \in \Gamma$ and some  $0 \le i < j$ such that $\tau x_{i|b|}=x_{j|b|}\,$. Without loss of generality (up to truncation of an initial piece of $\ray$), we assume $i=0$. Then for every $n \geq 1$, $\tau^n x_0= x_{nj|b|} \neq x_0$ and \cite[Prop. 12]{Ha2} yields that $\tau$ is non-elliptic.
\end{proof}

\begin{rmk}\label{rmk:torsion} In group theoretical terms the last lemma says
 that if the finitely generated group $\Gamma \leq \AUT(X,\ell)$ acts quasi-transitively on $X$ and is 
an infinite torsion group then $L_{\SAW}$ is \emph{not} context-free.
\end{rmk}

For the proofs of the next two propositions, we adopt the following notation:
if $w = \ell(\pi) \in \Si^*$ is the label of an arbitrary walk $\pi$ in $X$, then
$\bar w$ denotes the label of the reversed walk.

\begin{pro}\label{pro:nonreg} 
Let $X$ be a connected, infinite,  deterministically edge-labelled graph on which 
$\Gamma =\AUT(X,\ell)$ acts quasi-transitively.

\smallskip 

If $X$ contains a $\tau$-strip of size $2$, where $\tau \in \Gamma$, 
then $L_{SAW}$  is not regular.
\end{pro}

\begin{proof}
We suppose that $L_{SAW,o}(X)$ is regular for  $o \in V\!X$ and will reach a contradiction.

\smallskip

The strip has two $\tau$-invariant doubly infinite rays $\ray_1=[x_n\,, n \in \Z]$, $\ray_2=[y_n\,, n\in \Z]$ and there must be a shortest path $\pi$ connecting the two rays, which is contained in the connected subgraph $Y$ mentioned in the definition of $\tau$-strips. Without loss of generality,
we assume that $\pi$ goes from $x_0$ to $y_0$. 
Then the subgraph $X'$ of $X$ spanned by $\ray_1\,$, $\ray_2$ and all $\tau^{n}\pi$ ($n \in \Z$)
is a $\tau$-invariant subdivision of the bi-infinite ladder, see Figure \ref{fig:2ladder}.
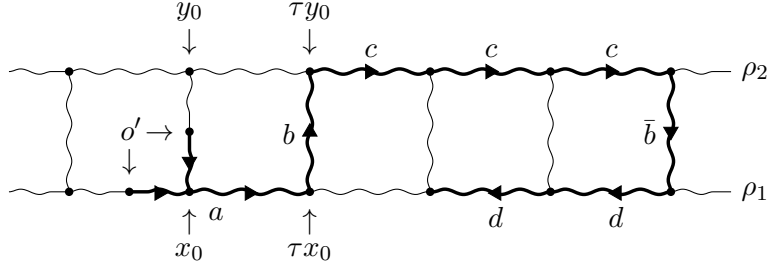
\begin{figure}[ht]
\pgfdecorationsegmentamplitude=1pt
\pgfdecorationsegmentlength=0.5cm
\centering
  	\begin{tikzpicture}[decoration={snake,segment length=5mm}]
  	\begin{scope}[scale=0.8]
	\draw[decorate] (-7,0) -- (-6,0);
	\draw[decorate] (-7,2) -- (-6,2);
	\draw[decorate] (4,0) -- (5,0);
	\draw[decorate] (4,2) --  (5,2);
	\draw[decorate] (-6,0) -- (-5,0);
	\draw[decorate] (-6,2) -- (-4,2);
	\draw[decorate] (-4,2) -- (-2,2);
	\draw[decorate] (-2,0) -- (0,0);
	\draw[decorate] (-6,0) -- (-6,2);
	\draw[decorate] (-4,2) -- (-4,1);
	\draw[decorate] (0,0) -- (0,2);
	\draw[decorate] (2,0) -- (2,2);
	\draw[decorate,line width=1.3pt] (4,0) -- (4,2);
	\draw[decorate,line width=1.3pt] (4,2) -- (2,2);
	\draw[decorate,line width=1.3pt] (2,2) -- (0,2);
	\draw[decorate,line width=1.3pt] (0,2) -- (-2,2);
	\draw[decorate,line width=1.3pt] (4,0) -- (2,0);
	\draw[decorate,line width=1.3pt] (2,0) -- (0,0);
	\draw[decorate,line width=1.3pt] (-2,0) -- (-2,2);
	\draw[decorate,line width=1.3pt] (-2,0) -- (-4,0);
	\draw[decorate,line width=1.3pt] (-4,0) -- (-4,1);
	\draw[decorate,line width=1.3pt] (-4,0) -- (-5,0);
%       \draw[-triangle 45] (2.95,0) -- node[yshift=12pt,xshift=2pt] {$a'$} +(-0.1,0);
    \draw[-triangle 60] (-2.95,-0.02) -- node[yshift=-8pt,xshift=-15pt] {$a$} +(0.1,0);
	\draw[-triangle 60] (-2,1.05) -- node[xshift=-8pt,yshift=-3pt] {$b$} +(0,0.1);
	\draw[-triangle 60] (-0.95,2) -- node[yshift=8pt,xshift=-2pt] {$c$} +(0.1,0);
	\draw[-triangle 60] (1.05,2) -- node[yshift=8pt,xshift=-2pt] {$c$} +(0.1,0);
	\draw[-triangle 60] (3.05,2) -- node[yshift=8pt,xshift=-2pt] {$c$} +(0.1,0);
	\draw[-triangle 60] (4,0.95) -- node[xshift=-8pt,yshift=2pt] {$\bar b$} +(0,-0.1);
	\draw[-triangle 60] (1.05,-0.02) -- node[yshift=-9pt,xshift=2pt] {$d$} +(-0.1,0);
	\draw[-triangle 60] (3.05,-0.02) -- node[yshift=-9pt,xshift=2pt] {$d$} +(-0.1,0); 
    \draw[-triangle 60] (-4,0.45) --  +(0,-0.1);
    \draw[-triangle 60] (-4.45,-0.02) -- +(0.1,0);
        \node ( ) at (-4.7,1.05) {$o'\! \rightarrow$};
        \node ( ) at (-5,0.5) {$\downarrow$};
%       \node ( ) at (-2.5,0.5) {$\searrow$};
        \node ( ) at (5.4,0) {$\ray_1$};
        \node ( ) at (5.4,2) {$\ray_2$};
        \node ( ) at (-4,-1) {$x_0$};
        \node ( ) at (-4,-0.5) {$\uparrow$};
        \node ( ) at (-2,-1) {$\tau x_0$};
        \node ( ) at (-2,-0.5) {$\uparrow$};
        \node ( ) at (-4,3) {$y_0$};
        \node ( ) at (-4,2.5) {$\downarrow$};
        \node ( ) at (-2,3) {$\tau y_0$};
        \node ( ) at (-2,2.5) {$\downarrow$};
	\foreach \j in {-3,...,1}{
		\fill (2*\j,2) circle (2pt);
		\fill (2*\j,0) circle (2pt);}
	\fill (4,0) circle (2pt);
	\fill (4,2) circle (2pt);
    \fill (-5,0) circle (2pt);
    \fill (-4,1) circle (2pt);
 	\end{scope}
  	\end{tikzpicture}
  	\caption{Labelled subdivision of the ladder}
  	\label{fig:2ladder}
\end{figure}
We can apply Lemma \ref{lem:subgraph} to $X'$, and there is $o' \in V\!X'$ such 
that $L_{\SAW,o'}(X')$ is regular. Without loss of generality, we assume that either
$o'$ lies on $\ray_1$ between $x_0$ and $\tau^{-1}x_0$ and is distinct from the latter, 
or that $o'$ lies on $\pi$ and is distinct from $y_0$.  (Otherwise we can exchange the two rays.) In Figure \ref{fig:2ladder}, we indicate the possible positions
of $o'$.

Let $a$ be the label of the path from $o'$ via $x_0$ to $\tau x_0\,$.
Write $b$ for the common label of all the paths $\tau^n \pi$ ($n \in \Z$)
 and $\bar b$ for the label of the reversed paths. Next, let $c$ denote the common label of each of 
the paths from any $\tau^n y_0$ to $\tau^{n+1} y_0$ within $\ray_2$.
And finally, let $d$ be the common label of the paths from any $\tau^{n+1}x_0$ to $\tau^n x_0$ 
within $\ray_1$ ($n \in \Z$). Each of the words $a,b,\bar b, c,d \in \Si^*$ is non-empty,
but they are in general not just elements of $\Si$.

The language consisting of all words
\begin{equation}\label{eq:words}
s=a\0 b\0 c^k\0 \bar b\0 d^l,\quad  k,l \in \N, 
\end{equation}
is regular, so that by the closure properties of regular languages,
also its intersection $\bar L$ with $L_{\SAW,o'}(X')$ is regular. 
Now, $\bar L$ consists of the labels of all self-avoiding walks starting at $o'$ 
which are of the form \eqref{eq:words}. Looking at Figure \ref{fig:2ladder}, 
such a walk goes from $o'$ to $\tau x_0\,$, then upwards to $\tau y_0$  and to the 
right to $\tau^{k+1}y_0$ along $\ray_2\,$, then downwards to $\tau^{k+1}x_0\,$,
and finally to the left to $\tau^{k+1-l}x_0\,$. Thus, in order to be self-avoiding, one must have 
$k > l$.

Let $p$ be the pumping length of Lemma \ref{lem:pumplemreg} for $\bar L$,
and let $z = a\0b\0 c^{p+1}\0 \bar b\0d^p \in \bar L$. In the decomposition
$z = u\0v\0\tilde u$ of the lemma, $|v\0\tilde u|\le p$ implies that $v\0\tilde u$ 
is a postfix of $d^p$.
That is, $d^p = \tilde v \0v\0\tilde u$ for some word $\tilde v \in \Si^*$. 
%and $\tilde v$ must be a prefix of $\tilde u$, 
Now also $w = u\0v^2\0\tilde u$ must be in $\bar L$, so that there must be $k, l$
such that
$$
a\0 b\0 c^k\0 \bar b\0 d^l = w = a\0b\0 c^{p+1}\0 \bar b\0 \tilde v\0v^2 \0\tilde u.
$$
Since the labelling is deterministic and the first symbol of $c$ and the first symbol of $\bar b$ are both labels of different edges starting at $y_0$, these symbols must be different. We can conclude that $k=p+1$ and $d^l =  \tilde v\0v^2 \0\tilde u$. This is longer than $d^p$, so that $l \ge p+1=k$.
But then the walk with label $w$ starting at $o'$ is not self-avoiding, a contradiction.
\end{proof}

\begin{pro}\label{pro:noncf} 
Let $X$ be as in Proposition \ref{pro:nonreg}.

\smallskip 

If $X$ contains a $\tau$-strip of size $3$, where $\tau \in \Gamma$, 
then $L_{SAW}$ is not context-free.
\end{pro}

\begin{proof} We suppose that $L_{\SAW,o}(X)$ is context-free and will again 
arrive at a contradiction.  
  
\smallskip

The strip contains three $\tau$-invariant rays $\ray_1=[x_n\,, n \in \Z]$, 
$\ray_2=[y_n\,, n\in \Z]$ and  $\ray_3=[z_n\,, n \in \Z]$. Up to renumbering the  
rays, exchanging them, and possibly replacing $\tau$ with a power of itself, we may assume 
to have the following situation: there is a path $\pi_1$ from $x_0$ to $y_0$ which meets
$\ray_1$ and $\ray_2$ only at its endpoints and does not meet $\ray_3\,$;  
there is a path $\pi_2$ from $y_r$ to $z_r$ ($r \ge 1$) which meets
$\ray_2$ and $\ray_3$ only at its endpoints and does not meet $\ray_1\,$, and furthermore,
$y_r$ lies strictly between $y_0$ and $\tau y_0\,.$ Then the subgraph $X'$ of $X$ spanned
by the three rays and all images $\tau^n \pi_1$ and $\tau^n \pi_2$ is a subdivision of
the bi-infinite $3$-ladder, see Figure \ref{fig:3ladder}.

 \begin{figure}[ht]
\pgfdecorationsegmentamplitude=1pt
\pgfdecorationsegmentlength=0.5cm
\centering
  	\begin{tikzpicture}[decoration=snake]
  	\begin{scope}[scale=1.2]
  	\foreach \j in {-2,...,0}{
		\draw[decorate] (2*\j,2) -- ({2*(\j+1)},2);
		\draw[decorate] (2*\j-1,0) -- ({2*(\j+1)-1},0);
		\draw[decorate] (2*\j+1,1) -- ({2*(\j+1)},1);
		\draw[decorate] (2*\j,1) -- (2*\j+1,1);
		\draw[decorate] (2*\j+1,0) -- (2*\j+1,1);
		\draw[decorate] (2*\j,1) -- (2*\j,2);}
	\draw[decorate,line width=1.3pt] (-6.5,1)--(-7,1);
	\draw[decorate,line width=1.3pt] (-8,1)--(-7,1);
	\draw[decorate,line width=1.3pt] (-8,1)--(-8,2);
	\draw[decorate,line width=1.3pt] (-8,2)--(-6,2);
	\draw[decorate,line width=1.3pt] (-6,2)--(-6,1);
	\draw[decorate,line width=1.3pt] (-6,1)--(-5,1);
	\draw[decorate,line width=1.3pt] (-5,1)--(-5,0);
	\draw[decorate,line width=1.3pt] (-5,0)--(-3,0);
	\draw[decorate,line width=1.3pt] (-3,0)--(-1,0);
	\draw[decorate,line width=1.3pt] (-1,0)--(1,0);
	\draw[decorate,line width=1.3pt] (1,0)--(1,1);
	\draw[decorate,line width=1.3pt] (1,1)--(2,1);
	\draw[decorate,line width=1.3pt] (2,1)--(2,2);
	\draw[decorate,line width=1.3pt] (0,2)--(2,2);
	\draw[decorate,line width=1.3pt] (-2,2)--(0,2);
	\draw[decorate,line width=1.3pt] (-4,2)--(-2,2);
	\draw[decorate,line width=1.3pt] (-4,1)--(-4,2);
	\draw[decorate,line width=1.3pt] (-4,1)--(-3,1);
	\draw[decorate,line width=1.3pt] (-3,1)--(-2,1);
	\draw[decorate,line width=1.3pt] (-2,1)--(-1,1);
	\draw[decorate,line width=1.3pt] (-1,1)--(0,1);
	\draw[decorate] (1,0) -- (2,0);
	\draw[decorate] (-7,0) -- (-5,0);
	\draw[decorate] (-6,2) -- (-4,2);
	\draw[decorate] (-5,1) -- (-4,1);
	\draw[decorate] (-8.5,2) -- (-8,2);
	\draw[decorate] (-8.5,1) -- (-8,1);
	\draw[decorate] (-7,1) -- (-6,1);
	\draw[decorate] (-8.5,0) -- (-7,0);
	\draw[decorate] (-7,0) -- (-7,1);
	\draw[decorate] (-8,1) -- (-8,2);
    \draw[decorate] (2,0) -- (2.5,0);
    \draw[decorate] (2,1) -- (2.5,1);
    \draw[decorate] (2,2) -- (2.5,2);
    \draw[decorate] (2,1) -- (2,2);
	\draw[-triangle 60] (-6.5,2) -- node[yshift=8pt,xshift=-3pt] {$a$} +(0.1,0);
	\draw[-triangle 60] (-5,0.5) -- node[xshift=-8pt,yshift=3pt] {$d$} +(0,-0.1);
	\draw[-triangle 60] (1,0.5) -- node[xshift=-8pt,yshift=-1pt] {$\bar d$} +(0,0.1);
	\draw[-triangle 60] (-3.95,0) -- node[yshift=-9pt,xshift=-3pt] {$e$} +(0.1,0);
	\draw[-triangle 60] (-1.95,0) -- node[yshift=-9pt,xshift=-3pt] {$e$} +(0.1,0);
	\draw[-triangle 60] (0.05,0) -- node[yshift=-9pt,xshift=-3pt] {$e$} +(0.1,0);
    \draw[-triangle 60] (1.55,1) -- node[yshift=9pt,xshift=-2pt] {$f$} +(0.1,0);
	\draw[-triangle 60] (2,1.5) -- node[xshift=8pt,yshift=-1pt] {$\bar b$} +(0,0.1);
	\draw[-triangle 60] (-2.95,2) -- node[yshift=8pt,xshift=2pt] {$g$} +(-0.1,0);
	\draw[-triangle 60] (-0.95,2) -- node[yshift=8pt,xshift=2pt] {$g$} +(-0.1,0);
	\draw[-triangle 60] (1.05,2) -- node[yshift=8pt,xshift=2pt] {$g$} +(-0.1,0);
%	\draw[-triangle 60] (-4,1.5) -- node[xshift=-8pt,yshift=-2pt] {$\bar f$} +(0,-0.1);
	\draw[-triangle 60] (-1.45,1) -- node[yshift=8pt,xshift=-2pt] {$c$} +(0.1,0);
	\draw[-triangle 60] (-3.45,1) -- node[yshift=8pt,xshift=-2pt] {$c$} +(0.1,0);
    \draw[-triangle 60] (-0.45,1) -- node[yshift=9pt,xshift=-2pt] {$f$} +(0.1,0);
	\draw[-triangle 60] (-2.45,1) -- node[yshift=9pt,xshift=-2pt] {$f$} +(0.1,0);
	\draw[-triangle 60] (-4,1.5) -- node[xshift=-8pt,yshift=2pt] {$b$} +(0,-0.1);
	\draw[-triangle 60] (-6,1.5) -- node[xshift=-8pt,yshift=2pt] {$b$} +(0,-0.1);
	\draw[-triangle 60] (-5.45,1) -- node[yshift=8pt,xshift=-2pt] {$c$} +(0.1,0);
% 	\draw[-triangle 60] (2,1.5) -- node[xshift=7.9pt,yshift=3pt] {$e$} +(0,0.1);
	\foreach \j in {-3,...,1}{
		\fill (2*\j,2) circle (1.5pt);
		\fill (2*\j-1,0) circle (1.5pt);
		\fill (2*\j,1) circle (1.5pt);
		\fill (2*\j-1,1) circle (1.5pt);}
	\fill (-8,2) circle (1.5pt);
	\fill (-8,1) circle (1.5pt);
	\fill (-8.01,1.5) circle (1.5pt);
    \fill (-6.5,1.01) circle (1.5pt);
    \fill (-7,2.01) circle (1.5pt);
    \node ( ) at (-6.95,1.55) {$o'$};
    \draw[->] (-7.15,1.5) -- (-7.9,1.5);
    \draw[->] (-6.9,1.4) -- (-6.6,1.1);
    \draw[->] (-7,1.65) -- (-7,1.9);
	\node ( ) at (2.7,0) {$\ray_3$};
	\node ( ) at (2.7,1) {$\ray_2$};
	\node ( ) at (2.7,2) {$\ray_1$};
	\node ( ) at (-5,-.65) {$z_r$};
	\node ( ) at (-5,-0.25) {$\uparrow$};
	\node ( ) at (-8,2.60) {$\tau^{-1} x_0$};
	\node ( ) at (-8,2.25) {$\downarrow$};
	\node ( ) at (-6,2.55) {$x_0$};
	\node ( ) at (-6,2.25) {$\downarrow$};
	\node ( ) at (-5,1.6) {$y_r$};
    \node ( ) at (-5,1.25) {$\downarrow$};
    \node ( ) at (-6,0.43) {$y_0$};
    \node ( ) at (-6,0.75) {$\uparrow$};
	\end{scope}
	\end{tikzpicture}
	\caption{Labelled subdivision of the 3-ladder}
	\label{fig:3ladder}
\end{figure}
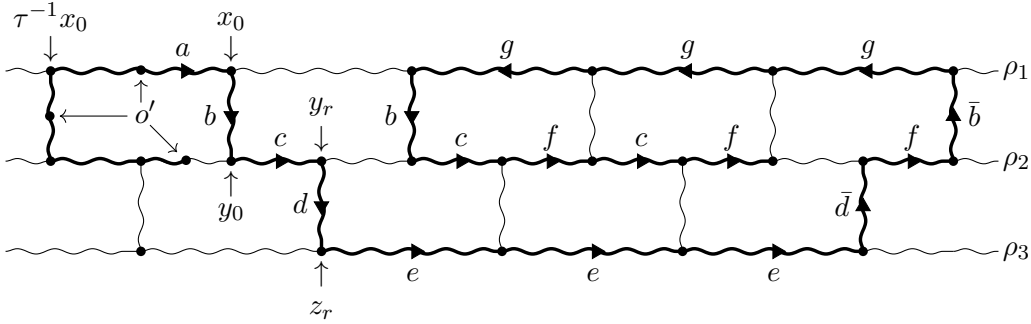

Again, Lemma \ref{lem:subgraph} applies to $X'$, and there is $o' \in V\!X'$ such 
that $L_{\SAW,o'}(X')$ is context-free. 

Up to possibly renumbering the rays, inverting their direction or exchanging
$\ray_1$ with $\ray_3\,$, we can assume without loss of generality that $o'$
lies on the ``rectangle'' with corners $x_0\,,\tau^{-1}x_0\,,\tau^{-1}y_0$ and $y_0\,$,
but not on the path $\pi_1$\,. 
In Figure \ref{fig:3ladder}, we  indicate the possible positions of $o'$.

We let $a \in \Si^*$ be the label of the self-avoiding walk starting at $o'$ and running around that rectangle in clockwise order up to $x_0\,$.
Write $b$ for the label of $\pi_1$ and $d$ for the label of $\pi_2$\,. 
%and $\bar b$ for the label of the reversed paths. 
Next, let $c$ and $f$ denote the labels of the subpaths of $\ray_2$ from $y_0$ to $y_r$ and from $y_r$ to $\tau y_0\,$, respectively.
Thereafter, $e$ denotes the label of the path from $z_r$ to $\tau z_r$ on $\ray_3$ and finally $g$ is the label of the path from $\tau x_0$ to $x_0$ on $\ray_1\,$. The automorphism $\tau$ is label preserving, any translates of the previous paths are also labelled by the same words.

Each of the words $a,b,\bar b, c,d, \bar d, e, f, g \in \Si^*$ is non-empty, but again, they are in general not just elements of $\Si$.

Similarly to the previous proposition, the language consisting of all words
\begin{equation}\label{eq:words2}
s=a\0 b\0 c\0 d\0 e^k\0 \bar d \0 f \0 \bar b\0 g^l\0 b\0 (cf)^m,\quad  k,l,m \in \N, 
\end{equation}
is regular, so that by the closure properties of context-free languages,
its intersection $\bar L$ with $L_{\SAW,o'}(X')$ is context-free.
Following the arrows in  Figure \ref{fig:3ladder}, one can see a self-avoiding 
walk with such a label, with $k=l=m+1=3$. In general, for a walk with label
$s$ as in \eqref{eq:words2} to be self-avoiding, one must have 
\begin{equation}\label{eq:3ladderineq}
k \ge l > m.
\end{equation}

Let $p$ be the pumping length of Lemma \ref{lem:pumplemcf} for $\bar L$,
and let $z = a\0 b\0 c\0 d\0 e^{p+1}\0 \bar d \0 f \0 \bar b\0 g^{p+1}\0 b\0 (cf)^p \in \bar L$.
In the decomposition $z = u\0v\0w\0\tilde v\0 \tilde u$ of the lemma,
%call the non-empty words among $v$ and $\tilde v$
%(there is at least one) the \emph{active part} of $v\0w\0\tilde v$. 
%The property 
$|v\0w\0\tilde v|\le p$ implies that 
$v\0w\0\tilde v$ is a subword of $a\0 b\0 c\0 d\0 e^{p+1}\0 \bar d\0 f\0 \bar b\0  g^{p+1}$
[Case 1], or of $g^{p+1}\0 b\0 (cf)^p$  [Case 2] (or both, meaning that it is contained in $g^{p+1}$.

In both cases, for any $n \ge 0$, there must be $k(n), l(n), m(n)$ such that
$$
z_n = u\0v^n\0w\0\tilde v^n\0 \tilde u 
= a\0 b\0 c\0 d\0 e^{k(n)}\0 \bar d \0 f \0 \bar b\0 g^{l(n)}\0 b\0 (cf)^{m(n)} \in \bar L.
$$
In \emph{Case 1}, $\tilde u$, and thus $z_n\,$, must end with $b \0(cf)^p$
Since the last symbol of $b$ and the last symbol of $f$ are labels of different edges ending at $y_0\,$, they have to be different. It follows that $m(n)=p$.
Using that $|z_0|<|z_1|$ we get that either $k(0)<k(1)=p+1$ or $l(0)<l(1)=p+1$, contradicting \eqref{eq:3ladderineq}. 

In \emph{Case 2}, we get in the same way that $k(n)=p+1$ for all $n$, and
either $l(2)>l(1)=p+1$ or $m(2)>m(1)=p$. Again this contradicts \eqref{eq:3ladderineq}.\end{proof}

We are now almost ready for the proof of Theorem \ref{thm:main-1}.
We will need Bass-Serre theory. As the topic cannot be briefly introduced, we will 
not give all definitions here. The reader is referred to 
{\sc Serre}~\cite{Se} and {\sc Dicks and Dunwoody}~\cite{DiDu}.

The ends of a finitely generated group are the ends of any of its Cayley
graphs with respect to a finite, symmetric set of generators. In fact,
they do not depend on the choice of the generating set; see Remark \ref{rmk:qi}.

A group  $\Gamma$ is called \emph{accessible} if it is the 
fundamental group of a finite graph of groups having finite edge groups and 
vertex groups which are finite or have one end. 
 The following lemma is \cite[Corollary IV.1.9]{DiDu}. 

\begin{lem}\label{lem:virtfree}
A group $\Gamma$ is the fundamental group of a finite graph of finite groups if and only 
if $\Gamma$ is finitely generated and virtually free. 
\end{lem}

A locally finite graph $X$ is called \emph{accessible}, if there is an integer $k$ 
such that any two ends of $X$ can be separated by a set containing $k$ or fewer vertices. 
{\sc Thomassen and Woess}~\cite{ThWo} showed that a connected, locally finite, 
transitive graph $X$ is accessible if and only if there is an integer $M$ 
such that each thin end of $X$ has size at most $M$. Moreover, they also proved 
that a finitely generated group is accessible if and only if some (and therefore 
all) of its Cayley-graphs are accessible.

\begin{proof}[\bf Proof of Theorem \ref{thm:main-1}]
First, $X$ cannot be one-ended. Indeed, in that case, that end has to be thick.
If $\Gamma=\AUT(X,\ell)$ has only elliptic elements, then $L_{\SAW}$ is not context-free by Lemma
\ref{lem:torsion}. Otherwise, $\Gamma$ has parabolic elements, and Lemma \ref{lem:parabolic}
combined with Proposition \ref{pro:noncf} implies as well that  $L_{\SAW}$ is not context-free.

Thus, $X$ has more than one end, whence there are thin ends. If $L_{\SAW}$ is context-free
then by Lemma \ref{lem:size} and Proposition \ref{pro:noncf} all thin ends have size at most $2$. We need to prove that there are no thick ends.

Recall from Lemma \ref{lem:fingen} and its proof that in the graph $X^{2D+1}$,
the orbit of $\Gamma o$ induces a Cayley graph $X(\Gamma, S)$ of the finitely generated 
group $\Gamma$. The identity mapping $\iota: \Gamma o \hookrightarrow V\!X$ induces 
a quasi-isometric embedding of $X(\Gamma,S)$ into $X$. Indeed, it is bi-Lipschitz 
and quasi-surjective, i.e., $B=0$ in \eqref{eq:qi}. Consequently, by Corollary
\ref{cor:qi}, all thin ends of $X(\Gamma,S)$ have size at most $\eta(2)$.
By \cite{ThWo}, the group $\Gamma$ is accessible. Thus, it is the fundamental group of a 
finite graph of finitely generated (sub)groups which are finite or one-ended.
If all of them are finite, then by Lemma \ref{lem:virtfree}, $\Gamma$ is virtually
free, so that it only has thin ends since $\Gamma$ is quasi-isometric with a tree.

Thus, if $\Gamma$ has a thick end, then it must have a finitely generated subgroup
$\Gamma_1$ which has one (thick) end.
Above, we have identified $\Gamma$ with the (vertex set of) $\Gamma o$ in $X$,
and $\Gamma_1 o$ is contained in that orbit. Under this identification, the group unit
corresponds to the ``root'' vertex $o$.
Let $S_1$ be a finite, symmetric set of generators of  $\Gamma_1$. Then for each 
$s \in S$ there is a (shortest)
path $\pi_s$ in $X$ from $o$ to the image $s\0 o$. We can choose these paths such
that $\pi_{s^{-1}}= s^{-1}\check \pi_s\,$, where $\check \pi_s$ is the reversal of $\pi_s\,$.
Let $V\pi_s$ and $E\pi_s$ be the sets of vertices and edges of $\pi_s\,$, respectively. Then we consider the subgraph $X_1$ of $X$ with
$$
V\!X_1= \bigcup_{\gamma \in \Gamma_1\0, s \in S} \gamma(V\pi_s) \AND
E\!X_1= \bigcup_{\gamma \in \Gamma_1\0, s \in S} \gamma(E\pi_s). 
$$
Clearly, $X_1$ is a connected subgraph of $X$ which inherits the labels from
the edges of $X$. Also, $X_1$ is quasi-isometric with the Cayley graph $X(\Gamma_1,S_1)$. 
Indeed, the embedding $\Gamma_1 \hookrightarrow V\!X_1\,$, $\gamma \mapsto \gamma o$
is bi-Lipschitz and quasi-surjective, i.e., $B=0$ in \eqref{eq:qi}. 

Therefore, $X_1$ has one end, which has to be thick, and it is quasi-transitive under
$\Gamma_1\,$. But then we are back to the situation of the beginning of this proof,
that is, $L_{\SAW,o}(X_1)$ cannot be context-free. But this contradicts Lemma 
\ref{lem:subgraph}. 

We conclude that $\Gamma$ and thus also $X$ have no thick ends.
\end{proof}

\section{Graphs with context-free language of SAWs}\label{sec:cf}
 The goal of this section is to prove the second half of Theorem \ref{thm:main}, namely
 
\begin{thm}\label{thm:main-2}
Let $X$ be a connected, locally finite, deterministically edge-labelled graph on which 
$\AUT(X,\ell)$ acts quasi-transitively. Then for every vertex $o$ of $X$ the following holds:
\begin{enumerate} \setlength\itemsep{0pt}
\item If all ends of $X$ have size 1, then $L_{SAW,o}$ is regular.
\item If all ends of $X$ have size at most 2, then $L_{SAW,o}$ is unambiguous context-free.
\end{enumerate}
\end{thm}

For an integer $k>0$ a graph $X$ is called $k$-connected if it has more than $k$ vertices and no set of less than $k$ vertices is a separating set in $X$. A \emph{(2-)block} in a graph $X$ is a maximal connected subgraph of $X$ containing no cutvertex. If $X$ is connected and has at least 2 vertices, every block of $X$ is either a pair of vertices connected by an edge or a 2-connected graph. The intersection of 2 different blocks of $X$ is either empty or a cutvertex in $X$. The \emph{block-cutvertex tree} $T_2(X)$ corresponding to $X$ is the graph having as vertices the blocks and the cutvertices of $X$, where a block is adjacent to every cutvertex it contains. Denote for an edge $e$ of $T_2(X)$ by $B(e)$ the block and by $c(e)$ the cutvertex incident to $e$. More about blocks and the block-cutvertex tree can be found in \cite{Tu}. \par

%Let $X$ be a graph and $Y,Z$ be two subgraphs of $X$. We call $Y$ and $Z$ $AUT(X)$-equivalent, if there is some $\gamma \in Aut(X)$ mapping $Y$ to $Z$. In particular two vertices $u$ and $v$ are $AUT(X)$-equivalent, if there is a $\gamma \in AUT(X)$ with $\gamma u=v$. \par 

If $X$ is locally finite, every cutvertex of $X$ belongs to a finite number of blocks and therefore has finite degree in $T_2(X)$. On the other hand an infinite block of $X$ can contain infinitely many cutvertices, so $T_2(X)$ need not be locally finite.\par 
For any automorphism $\gamma \in \AUT(X)$, the image $\gamma B$ of a block $B$ of $X$ is again a block of $X$ and the same holds for cutvertices. Therefore, whenever a subgroup $\Gamma \le \AUT(X)$ acts quasi-transitively on $X$, it also acts quasi-transitively on the graph $T_2(X)$. \par

The following lemma shows that blocks of quasi-transitive graphs are always quasi-transitive graphs.

\begin{lem}\label{lem:blocksqt}
Let $X$ be a connected, locally finite graph and suppose $\Gamma \le \AUT(X)$ acts quasi-transitively on  $X$. Then for any block $B$ of $X$, the set-wise stabilizer $\Gamma_B$ of $B$ in $\Gamma$ acts quasi-transitively on $B$.
\end{lem}
\begin{proof}
$\Gamma$ acts quasi-transitively on $X$, so it acts with finitely many orbits on $E\!B$, i.e., the set $\{\Gamma e \mid e \in E\!B\}$ is finite. But every $\gamma \in \Gamma$ mapping some edge of $B$ onto another edge of $B$ clearly fixes the block $B$ and is therefore also contained in $\Gamma_B$. This implies that also $\Gamma_B$ acts with finitely many orbits on $E\!B$ and thus quasi-transitively on $B$.
\end{proof}

The following lemma is a simple consequence of the fact that a block cannot contain ends of size 1, because defining sequences of such ends consist of cutvertices.

\begin{lem}\label{lem:blocksfinite}
Let $X$ be a locally finite graph such that all ends of $X$ are of size 1. Then every block of $X$ is finite.
\end{lem}
%\begin{proof}
%Let $B$ be a 2-connected block and assume $B$ is infinite.  Let $\pi$ be a ray in $B$ and $\omega$ be the end of $X$ defined by $\pi$. Then the size of $\omega$ is equal to 1, so there is a defining sequence $K_0, K_1, \dots$ of $\omega$ consisting of sets of cardinality 1. By definition each of these sets consists of exactly one cut-vertex of $X$ and $\pi$ contains infinitely many of those sets. Let $K_i$ separate two vertices $u$ and $v$ contained in $\pi$. Then $K_i$ separates $u$ and $v$ in $B$. This contradicts the assumption of $B$ being 2-connected.
%\end{proof}

In the case where the graph $X$ has infinite blocks, we want to further decompose them. There are different ways to do this. One natural way is to use 3-block-decompositions, first introduced by {\sc Tutte} (see \cite{Tu}) for finite graphs and then generalized to infinite graphs by {\sc Droms, B. Servatius and H. Servatius} in \cite{DSS}. In this theory, sometimes graphs may have multiple edges between a single pair of vertices, we will call them multi-graphs. \par

Let $X$ and $Y$ be two multi-graphs and let $e \in E\!X$ and $f \in EY$ be (directed) edges. The edge amalgam of $X$ and $Y$ along $e$ and $f$ is the graph $Z$ obtained from the disjoint union of $X$ and $Y$ by identifying the vertices $e^-$ with $f^-$, $e^+$ with $f^+$ and erasing the edges $e$ and $f$. A convenient way to represent a sequence of edge amalgamations of a (not necessarily finite) set of multi-graphs is the edge-amalgam tree $T$. Vertices of $T$ are the multi-graphs used in the amalgamation. For clarity we will denote by $A(\alpha)$ the multi-graph corresponding to the vertex $\alpha$ of $T$. Two vertices $\alpha$ and $\beta$ are connected by an undirected edge in $T$, if $A(\alpha)$ and $A(\beta)$ are amalgamated along some edges $e \in E\!A(\alpha)$ and $f \in E\!A(\beta)$. We additionally introduce a label function $\lambda$ assigning to every directed edge $a \in ET$ the edge $e$ of $A(a^-)$ used in the amalgamation of $A(a^-)$ and $A(a^+)$. For $\alpha \in VT$ an edge of  $A(\alpha)$ is called \emph{virtual}, if it is the label of some edge of $T$, otherwise it is called \emph{non-virtual}. Denote the resulting multi-graph obtained from a sequence of edge amalgamations given by an amalgamation tree $T$ by $A(T)$. Then virtual edges disappear during the progress, while non-virtual edges are still present in $A(T)$.\par

A \emph{multilink} is a multi-graph consisting of 2 vertices and a (finite) positive number of undirected edges between these vertices. A multi-graph is said to be a \emph{3-block} if it contains at least 3 edges and is either a cycle (closed path), a multilink or a locally finite 3-connected graph. 
An edge-amalgam tree $T$ is called a \emph{3-block tree} if for every $\alpha \in VT$, $A(\alpha)$ is a 3-block and additionally for every edge $a \in ET$ the corresponding graphs $A(a^-)$, $A(a^+)$ are neither both multilinks nor both cycles. Note that in general, 3-block trees need not be locally finite. This is due to the fact that there can be infinite 3-blocks and therefore also infinitely many amalgamations of a single 3-block. \par

\begin{thm}[{{\sc Droms, B. Servatius, H. Servatius} \cite[Thm. 1]{DSS}}] \label{thm:3blocktree}
	For any locally finite, 2-connected graph $X$ there is a unique 3-block tree $T$ such that $X=A(T)$.
\end{thm}

For a given 2-connected graph $X$ the unique 3-block tree given by the above theorem will henceforth be denoted by $T_3(X)$. The proof of the theorem is constructive and allows us to ``decompose" $X$ in a unique way into (possibly infinitely many) 3-blocks, such that $X$ is obtained from amalgamating these 3-blocks as given by $T_3(X)$. The set of vertices and the set of non-virtual edges of each 3-block will be considered as a subsets of $V\!X$ and $E\!X$, respectively. A single vertex of $X$ may appear in different blocks. \par

$T_3(X)$ has the following 2 properties:
\begin{enumerate}\setlength\itemsep{0pt}
	\item[(a)] For every virtual edge $e=\lambda(a)$, $a \in ET_3(X)$, there is a finite sub-tree $T'$ of $T_3(X)$ containing $a^-$ but not $a^+$ and a path in $A(T')$ connecting the endpoints of $e$ and consisting of edges of $X$.
	\item[(b)] Let $\alpha \in VT_3(X)$ and $v \in V\!A(\alpha)$. Then $\alpha$ is contained in a finite sub-tree $T'$ of $T_3(X)$ such that all edges of $A(T')$ incident to $v$ are edges of $X$.
\end{enumerate} 

Due to the uniqueness of the decomposition, symmetries on the graph $X$ carry over to $T$ in a canonical way. Moreover, as in the case of 2-blocks, any $\Gamma \le \AUT(X)$ acting quasi-transitively on $X$ also acts quasi-transitively on $T_3(X)$.

\begin{lem}\label{lem:3blocksfinite}
Let $X$ be a simple, locally finite, 2-connected graph such that all ends of $X$ are of size at most 2.  Then every 3-block of $X$ is finite.
\end{lem}
\begin{proof}
Clearly, multilinks and cycles are finite. Every end of a 3-connected graph must be of size at least 3, because every defining sequence of an end consists of separating sets.\par
Assume that there is an infinite 3-block $A$ of $X$. Then it contains an end $\omega$ of size at least 3 and this end contains 3 disjoint rays. By property (a) of $T_3(X)$, we can replace all virtual edges contained in the rays by finite paths consisting of non-virtual edges of $X$ which are not in $A$. We obtain 3 disjoint rays in $X$, which belong to the same end of $X$. This is a contradiction to the assumption that all ends of $X$ have size at most 2. 
\end{proof}

\begin{thm}\label{lem:2conlangcf}
Let $X$ be a 2-connected, locally finite, deterministically edge labelled graph and let $\Gamma \le \AUT(X,\ell)$ act quasi-transitively on $X$. If every end of $X$ is of size at most 2, then for every $o \in V\!X$, the language $L_{SAW,o}$ is unambiguous context-free.
\end{thm}
\begin{proof}
Let $R$ be a set of representatives of the finite set of orbits $\Gamma \backslash ET_3(X)$ of directed edges of $T_3(X)$. For $a \in ET$ and a vertex $u \in \lambda(a)$ write $\bar{a}$ for the representative of $a$ in $R$ and $\bar{u}$ for the vertex in $\lambda(\bar{a})$ representing $u$.  Define for $a \in R$ and $u \in \lambda(a)$ the set $\Pi_{a,u}^0$ of self-avoiding walks in $A(a^+)$ starting at $u$ and not containing the virtual edge $\lambda(\check a)$. Let $\Pi_{a,u}^1$ be the subset of $\Pi_{a,u}^0$ consisting of all walks not containing the second vertex of $\lambda(a)$. Note that both sets are finite because $A(a^+)$ is finite. Fix some vertex $\alpha_o$ of $T_3(X)$ such that the 3-block $A(\alpha_o)$ contains the vertex $o$ of $V\!X$. We denote by $\Pi_o$ the set of self-avoiding walks in $A(\alpha_o)$ starting at $o$.

We extend the label function $\ell$ on $X$ to 3-blocks of $X$. Labels of non-virtual edges are inherited from $X$ and for any $a \in ET_3(X)$ label $e=\lambda(a)$ by $U_{\bar{a},\bar{e}^-}$ and $\check e$ by $U_{\bar{a},\bar{e}^+}\,$. The extended label function will be again denoted by $\ell$ and maps into $\Si \cup \Si'$, where  $\Si$ is the label alphabet of $X$ and $\Si'=\{U_{a,u} \mid a \in R,\; u \in \lambda(a)\}$.

We will now present a grammar $\Ccal=(\V,\Si,\Pb,S)$ generating the language of self-avoiding walks in $X$ starting at $o$. The finite set of variables is given by 
\begin{equation*}
\V= \{S\} \cup \Si' \cup \{V_{a,u}^i \mid a \in R, u \in \lambda(a), i \in \{0,1\}\}.
\end{equation*}
Let $a \in R$ and $u \in \lambda(a)$. Then for any directed edge $b \neq \check a$ in $ET_3(X)$ starting at $a^+$ the productions given below are contained in $\Pb$.
\begin{alignat*}{2}
	&V_{a,u}^0 \vdash \ell(\pi) \quad &&\parbox[t]{25em}{for every $\pi \in \Pi_{a,u}^0$ ending with a non-virtual edge,} \\[4pt]
	&V_{a,u}^0 \vdash \ell(\pi) V_{\bar{b},\bar{v}}^0 \quad &&\parbox[t]{25em}{for every $\pi \in \Pi_{a,u}^0$ ending at a vertex $v \in \lambda(b)$ and not containing the second vertex of $\lambda(b)$,} \\[4pt]
	&V_{a,u}^0 \vdash \ell(\pi) V_{\bar{b},\bar{v}}^1 \quad &&\parbox[t]{25em}{for every $\pi \in \Pi_{a,u}^0$ ending at a vertex $v \in \lambda(b)$ and containing the second vertex of $\lambda(b)$, but not $\lambda(b)$,} \\[4pt]
	&V_{a,u}^1 \vdash \ell(\pi) \quad &&\parbox[t]{25em}{for every $\pi \in \Pi_{a,u}^1$ ending with a non-virtual edge,} \\[4pt]
	&V_{a,u}^1 \vdash \ell(\pi) V_{\bar{b},\bar{v}}^0 \quad &&\parbox[t]{25em}{for every $\pi \in \Pi_{a,u}^1$ ending at a vertex $v \in \lambda(b)$ if the second vertex of $\lambda(b)$ is  neither contained in $\pi$ nor in $\lambda(a)$,} \\[4pt] 
	&V_{a,u}^1 \vdash \ell(\pi) V_{\bar{b},\bar{v}}^1 \quad &&\parbox[t]{25em}{for every $\pi \in \Pi_{a,u}^1$ ending at a vertex $v \in \lambda(b)$ if the second vertex of $\lambda(b)$ is either contained in $\pi$ or in $\lambda(a)$,} \\[4pt]
	&U_{a,u} \vdash \ell(\pi) \quad &&\parbox[t]{25em}{for every $\pi \in \Pi_{a,u}^0$ ending at the second vertex of $\lambda(a)$.}
\end{alignat*}
The set of productions $\Pb$ is completed by adding for every edge $b \in ET_3(X)$ starting at $\alpha_o$ the following rules.
\begin{alignat*}{2}
	&S \vdash \ell(\pi) \quad &&\parbox{25em}{for every $\pi \in \Pi_o$ ending with a non-virtual edge,} \\[4pt]
	&S \vdash \ell(\pi) V_{\bar{b},\bar{v}}^0 \quad &&\parbox[t]{25em}{for every $\pi \in \Pi_o$ ending at a vertex $v \in \lambda(b)$ and not containing the second vertex of $\lambda(b)$,} \\[4pt] 
	&S \vdash \ell(\pi) V_{\bar{b},\bar{v}}^1 \quad &&\parbox[t]{25em}{for every $\pi \in \Pi_o$ ending at a vertex $v \in \lambda(b)$ and containing the second vertex of $\lambda(b)$ but not $\lambda(b)$.} 
\end{alignat*}

We will now briefly discuss why the given grammar $\Ccal$ unambiguously generates the language of self-avoiding walks in $X$ starting at $o$.

Let $a=[\alpha,\beta]$ be an edge of $T_3(X)$ and $T'$ be the component of $T_3(X) \setminus \{\alpha\}$ containing $\beta$. A self-avoiding walk $\pi$ of length at least 1 in $X$ is called a V-walk with direction $a$ if it starts at a vertex $u \in \lambda(a)$ and all edges of the walk are contained in $A(T')$. A U-walk with direction $a$ is a V-walk with direction $a$ also ending in $\lambda(a)$.

Then the following statements hold.
\begin{enumerate}
\item[(a)] For $a \in ET$ and $u \in \lambda(a)$, the variable $U_{\bar{a},\bar{u}}$ unambiguously generates the language of all U-walks with direction $a$ starting at $u$.
\item[(b)] For $a \in ET$ and $u \in \lambda(a)$, the variable $V_{\bar{a},\bar{u}}^0$ unambiguously generates the language of all V-walks with direction $a$ starting at $u$, and the variable $V_{\bar{a},\bar{u}}^1$ unambiguously generates the language of all V-walks with direction $a$ starting at $u$ and not containing the second vertex of $\lambda(a)$.
\item[(c)] The start symbol $S$ unambiguously generates the language $L_{SAW,o}\,$.
\end{enumerate}

Rigorous proofs for these statements are long and technical, so we only sketch them here.

Let $\alpha$ be a vertex of $T$ and $A(\alpha)$ be its corresponding 3-block. Define the projection $\pi(\alpha)$ of a self-avoiding walk $\pi$ onto $A(\alpha)$ in the following way: Let $v_1, \dots, v_k$ be the sequence of vertices of $\pi$ which are contained in $A(\alpha)$ ordered by their occurrence in $\pi$ and $\pi_i$ be the subwalk of $\pi$ connecting $v_i$ and $v_{i+1}\,$. For every $i$ there are 2 cases: If $\pi_i$ is a single edge of $A(\alpha)$, add this edge to $\pi(\alpha)$. Otherwise there is an edge $a$ of $T$ such that $\pi_i$ is a U-walk with direction $a$ and we add the virtual edge $\lambda(a)$ connecting $v_i$ and $v_{i+1}\,$. This edge can be seen as a shortcut for the U-walk $\pi_i$. For reasons of ambiguity, if $\pi_{k-1}$ is a U-walk and $\pi$ ends at $v_k$, we do not add the corresponding virtual edge and consider our projection to end at $v_{k-1}\,$. The resulting $\pi(\alpha)$ is a self-avoiding walk in $A(\alpha)$. \par

For every U-walk $\pi$ with direction $a$ starting at $u$ we can obtain the word $\ell(\pi(a^+))$ corresponding to the projection of $\pi$ onto the 3-block $A(a^+)$ from the variable $U_{a,u}$ in a single derivation step. A simple induction shows that $\pi$ is generated by $U_{a,u}\,$. Moreover for any word $w$ generated by $U_{a,u}\,$, the walk $\pi$ starting at $u$ and having label $w$ is a U-walk with direction $a$ starting at $u$. The word $w$ can only be obtained from a unique sequence of rightmost derivations because in every step we have to generate the string corresponding to the projection of $\pi$ onto some 3-block. Statement (a) follows.\par

For any V-walk $\pi$ with direction $a$ starting at $u$ we can derive $\ell(\pi(a^+))$ in a single step of derivation from $V_{\bar{a},\bar{u}}^0$ and if $\pi$ does not contain the second vertex of $\lambda(a)$, also from $V_{\bar{a},\bar{u}}^1$. Using (a) and induction, this implies that $\ell(\pi)$ is generated by the corresponding variables.  Furthermore note that $V_{\bar{a},\bar{u}}^0$ only appears in the right hand side of productions, if the second vertex of $\lambda(a)$ is not contained in the projection of $\pi$ on any block previously visited by $\pi$. Using this observation it is not hard to show that walks corresponding to words generated by $V_{\bar{a},\bar{u}}^0$ and $V_{\bar{a},\bar{u}}^1$ are indeed V-walks with direction a starting at $u$ and that every such walk is generated unambiguously. We obtain (b).

Finally, for every SAW $\pi$ starting at $o$ we can derive $\ell(\pi(\alpha_o))$ in a single step from $S$ and (c) follows from (a) and (b) as before.
\end{proof}

We are now able to prove the main result in this section. This result together with Theorem \ref{thm:main-1} then implies Theorem \ref{thm:main}.

\begin{proof}[\bf Proof of Theorem \ref{thm:main-2}]
The group $\Gamma=\AUT(X,\ell)$ acts quasi-transitively on the block-cutvertex tree $T_2(X)$. Let $Y= \Gamma \backslash T_2(X)$ denote the finite, connected factor graph. Every vertex of $Y$ corresponds to either a class of blocks or a class of cutvertices of $X$ under the action of $\Gamma$, so for $e \in EY$ we can again write $B(e)$ (block) and $c(e)$ (cutvertex) for the two vertices of $Y$ incident to $e$. Note that for any cutvertex $c$ of $X$, all edges of $T_2(X)$ incident to $c$ lie in different orbits with respect to $\Gamma$ because by Lemma \ref{lem:fingen}, $\Gamma$ acts fixed-point-freely on $X$. For $e \in EY$ let $N(e)=\{f\in EY\setminus \{e\} \mid c(e)=c(f)\}$. Fix some block $B_o$ containing the vertex $o$.\par

Let $L$ be the regular language generated by the grammar $\Ccal=(\V,\Si,\Pb,S)$, where 
\begin{alignat*}{3}
&\V &&=\;&&\{S\} \cup \{W_e \mid e \in EY\}, \\[4pt]
&\Si &&=&& \{\alpha_o\}  \; \cup\; \{\alpha_{o,f} \mid f\in EY\} \; \cup \; \{\alpha_e \mid e \in EY\} \; \cup \; \{\alpha_{e,f} \mid e,f \in EY\},
\end{alignat*}
and the set of productions $\Pb$ consists of
\begin{alignat*}{3}
	&S &&\vdash  \alpha_{o} \,,\\[4pt]
	&S &&\vdash \alpha_{o,f} W_g \quad &&\parbox[t]{25em}{for $f \in EY$, $g \in N(f)$,} \\[4pt]
	&W_e &&\vdash \alpha_e  \quad &&\parbox[t]{25em}{for $e \in EY$, } \\[4pt] 
	&W_e &&\vdash \alpha_{e,f} W_g \quad &&\parbox[t]{25em}{for $e,f \in EY$, $g \in N(f)$.} 
\end{alignat*}

For $e \in EY$ let $L(e)$ be the language of self-avoiding walks of length at least 1 in the block $B(e)$ starting at a fixed representative $c$ of the vertex $c(e)$, seen as a class of vertices of $X$. Clearly, $L(e)$ does not depend on the choice of $c$. By Lemma \ref{lem:blocksqt} the stabilizer $\Gamma_{B(e)}$ acts quasi-transitively on the graph $B(e)$ and by assumption on $X$ all ends of $B(e)$ have size at most 2. Hence Lemma \ref{lem:2conlangcf} applies and $L(e)$ is context-free. Denote for $f \in EY$ by $L(e,f)$ the subset of $L(e)$ corresponding to walks ending at vertices of $B(e)$, which are in the vertex class $c(f)$. Note that $L(e,f)$ may be an empty language if $B(e)\neq B(f)$ or if $e=f$ and $c$ is the only representative of $c(e)$ in $B(e)$. As the intersection of the unambiguous context-free language $L(e)$ and the regular language of all walks starting at $c$ and ending at a representative of $c(f)$, $L(e,f)$ is unambiguous context-free. In a similar way let $L(o)$ be the language of all walks in $B_o$ starting at $o$ and $L(o,f)$ be the subset of $L(o)$ corresponding to walks ending at a representative of $c(f)$.\par 

Let $\varphi$ be the substitution of languages given for $e,f \in EY$ by
\begin{equation*}
\varphi(\alpha_{o})= L(o), \quad \varphi(\alpha_{o,f})=L(o,f), \quad \varphi(\alpha_{e})=L(e), \quad \varphi(\alpha_{e,f})=L(e,f).
\end{equation*}
Then by  \cite[Thm. 3.4.1]{Ha} the result $\varphi(L)$ of the substitution is context-free. 

If every end of $X$ is of size at most 1, by Lemma \ref{lem:blocksfinite} every block of $X$ is finite and thus also the language of self-avoiding walks in the block is finite. We conclude that in this case $\varphi(L)$ is regular.

For $e \in ET_2(X)$, a self-avoiding walk $\pi$ of length at least 1 in $X$ is called a W-walk with direction $e$ if it starts at $c(e)$ and its first edge lies in the block $B(e)$. Then the following statements hold:
\begin{enumerate}
\item[(a)] The variable $W_e$ generates a regular language $L_e$ such that $\varphi(L_e)$ is the language of W-walks with direction $e$.
\item[(b)] $\varphi(L)$ is the unambiguous context-free language $L_{SAW,o}$ of self-avoiding walks in $X$ starting at $o$.
\end{enumerate}
As before we will only sketch the proofs for these statements. For $e \in ET_2(X)$ we denote by $\bar e$ the orbit of $e$ under $\Gamma$, which is an edge of $Y$.

Let $\pi$ be a W-walk with direction $e \in ET_2(X)$. Let $\pi_1$ be the part of $\pi$ contained in the block $B(e)$. If $\pi_1=\pi$, then $\ell(\pi_1)$ is contained in $L(\bar e)$ and therefore obtained in a single step of derivation. Otherwise there is some $f \in ET_2(X)$ such that $\pi$ leaves $B(e)$ via $c(f)$ and $\ell(\pi_1)$ is contained in $L(\bar e,\bar f)$. In this case $\pi$ enters one of the other blocks containing $c(f)$, which are blocks $B(g)$, $g \in ET_2(X)$, such that $\bar g \in N(\bar f)$. The rest of $\pi$ is a W-walk with direction $g$. A simple induction shows that $\ell(\pi)$ is contained in $\varphi(L_e)$. 

On the other hand, every word $w \in \varphi(L_e)$ corresponds to a unique walk $\pi$ starting at $c(e)$ labelled by $w$. This walk $\pi$ is self-avoiding, because the parts of $\pi$ contained in blocks are self-avoiding, and whenever leaving a block $B(f)$, $f \in ET_3(X)$, $\pi$ can never enter $B(f)$ again because $T_2(X)$ is a tree and $N(\bar f)$ does not contain $\bar f$. This implies that $\pi$ is a W-walk with direction $e$ and proves (a).

In the same way it can be seen that for any $\pi \in \Pi_{SAW,o}\,$, the label of the part $\pi_0$ contained in $B_o$ is contained in $L(o)$ if $\pi_0=\pi$ and in $L(o,f)$, if $\pi$ leaves $B_o$ via a vertex in the class $c(f)$. Therefore $\ell(\pi)$ is contained in $\varphi(L)$. 
In the same way as in (a) we obtain that every walk starting in $o$ and corresponding to a word in $\varphi(L)$ is self-avoiding. 

Let $w \in \varphi(L)$ and $\pi$ be the SAW starting at $o$ and having label $w$. There is a unique word $w' \in L$ such that $w \in \varphi(w')$, which is given by the sequence of blocks visited by $\pi$. Furthermore, for every $a \in \Si$ the language $\varphi(a)$ is unambiguous context-free, hence $\varphi(L)$ is also unambiguous context-free. This shows statement (b) and finishes the proof.
\end{proof}

\section{Discussion and examples}\label{sec:final}
The proof of Theorem \ref{thm:main-2} is constructive and the obtained grammar can be used to calculate the generating function of self-avoiding walks $F_{SAW}(t|o)$ and the connective constant of graphs satisfying the conditions of the theorem.\par 
Given some language $L$, the ordinary generating function $F_L(t)$ is the power series
\begin{equation*}
	F_L(t)= \sum_{w\in L} t^{|w|}.
\end{equation*}

Using the algebraic theory of context-free languages {\sc Chomsky and Sch\"utzenberger} developed in \cite{ChSch}, the productions of an unambiguous context-free grammar $\Ccal$ generating the language $L$ can be translated into a system of polynomial equations having as one of its solutions the language generating function $F_L(t)$.

Recall that a power series $F(t)$ is called algebraic over a field $K$, if it satisfies a polynomial equation of the form $P\bigl(t,F(t)\bigr)=0$, where $P(x,y)$ is a bivariate polynomial in $K[x,y]$. From classical elimination theory it follows that any component of a solution of a system of polynomial equations having coefficients in $K$ is algebraic over $K$, in particular $F_L(t)$ is algebraic over $\mathbb{Q}$.

Let $X$ be a connected, locally finite, deterministically edge labelled graph and $o$ be a vertex of $X$ such that the language $L_{SAW,o}$ is unambiguous context-free. Then the label function $\ell$ acts as a bijection between the set $\Pi_{SAW,o}$ of self-avoiding walks starting at $o$ and its language $L_{SAW,o}\,$, whence the SAW-generating function satisfies $F_{SAW}(t|o)=F_{L_{SAW,o}}(t)$. All singularities of algebraic functions are algebraic numbers, so in particular the radius of convergence of $F_{SAW}(t|o)$ and thus also the connective constant of the graph $X$ are algebraic numbers. 

Note that there are quasi-transitive graphs $X$ which do not admit any deterministic labelling $\ell$ such that $\AUT(X,\ell)$ acts quasi-transitively on $X$. A simple example is the Grandparent Graph $GP$, given by {\sc Trofimov} in \cite{Tr} as a graph with a non-unimodular automorphism group $\Gamma$. As a consequence, every subgroup of $\Gamma$ acting quasi-transitively on $GP$ cannot act fixed point freely and by Lemma \ref{lem:fingen}, $GP$ admits no labelling as above. 
Nevertheless, the following statement can be shown using the previous discussions and the ideas and proofs of Section \ref{sec:cf}, but generating functions counting walks instead of grammars and language theory. 
\begin{cor}
Whenever all ends of a connected, locally finite, quasi-transitive graph $X$ are of size at most 2, the SAW-generating function $F_{SAW}(t|o)$ is algebraic over $\mathbb{Q}$. In particular the connective constant $\mu(X)$ is an algebraic number.
\end{cor}

Alm and Janson showed in \cite{AJ} that the generating functions of self-avoiding walks on one-dimensional lattices are algebraic over $\mathbb{Q}$, independently of the size of the ends. In future work, we shall examine how our results can be extended to quasi-transitive graphs having only thin ends. A further issue for future work is to investigate under which structural conditions on the graph, $L_{SAW}$ is accepted by a multipass automaton as in \cite{CCFST}.

Some concrete examples, where we used Theorem \ref{thm:main-2} and its constructive proof to obtain the SAW-generating functions and the connective constants, include the ``sandwich" of two $k$-regular trees, which was already treated in a slightly different way in \cite{Li}, the above mentioned Grandparent Graph and the Cayley graph of the group $\Gamma=\langle a,c \mid c^2=1 \rangle$ with respect to the generators $S=\{a,b=ca,c\}$, which has non-transitive blocks.\par
Finally we provide the following example, where we end up with an algebraic SAW-generating function, which is not rational.

\begin{exa}
Consider the group $\Gamma=\langle a,b,c \mid a^2=b^2=c^2=(ab)^3=(bc)^2=1 \rangle$, which is a free product with amalgamation of the dihedral groups $D_3$ and $D_2$ and let $X$ be the Cayley graph $X\!\left(\Gamma,\{a,b,c\}\right)$. Edges are labelled by the generators in the usual way. The resulting graph can be seen in Figure \ref{fig:caleygraphgroupac}. All ends of $X$ have size 2, so the language of self-avoiding walks in $X$ is context-free. The 3-block-decomposition of $X$ can be seen in Figure \ref{fig:caleygraphgroupac}, where virtual edges are dashed. It yields 3 types of 3-blocks, which will be denoted by $A$, $B$ and $C$ in correspondence to the labels they contain. $ET_3(X)$ contains 4 types of edges, they will be denoted by $AB$, $BA$, $BC$, $CB$, depending on the pair of 3-blocks they connect (e.g. $AB$ starts at $A$ and ends at $B$). \par
\begin{figure}[ht]
\pgfdecorationsegmentamplitude=0.5pt
\pgfdecorationsegmentlength=0.5cm
\centering
\begin{subfigure}{0.50\textwidth}
	\centering
	\begin{tikzpicture}
	\begin{scope}[scale=0.9]
	\foreach \i in {0,1,2} {
		\draw (120*\i:1)--(120*\i+60:1) -- (120*\i+120:1);
		\node at (120*\i-30:1.1) {$a$};
		\node at (120*\i+30:0.6) {$b$};
		\draw (120*\i+30:{1+sqrt(3)})+(120*\i+180:1) --(120*\i+60:1);
		\node at ($(120*\i+30:{(1+sqrt(3))/2})+(120*\i+120:0.7)$) {$c$};
		\node at ($(120*\i+30:{(1+sqrt(3))/2})+(120*\i-60:0.7)$) {$c$};
		\draw (120*\i+30:{1+sqrt(3)})+(120*\i+240:1) --(120*\i:1);
		\foreach \j in {0,1,2} {
			\node at ($(120*\i+30:{1+sqrt(3)})+(120*\i+120*\j+30:1.1)$) {$a$};
			\node at ($(120*\i+30:{1+sqrt(3)})+(120*\i+120*\j-30:0.6)$) {$b$};
			\draw (120*\i+30:{1+sqrt(3)})++(120*\i+120*\j:1) -- ++(120*\i+120*\j+120:1);
			\draw (120*\i+30:{1+sqrt(3)})++(120*\i+120*\j:1) -- ++(120*\i+120*\j-120:1);
		}
		\draw[dashed] (120*\i+30:{1+sqrt(3)})++(120*\i:1) -- ++(120*\i-30:0.5);
		\draw[dashed] (120*\i+30:{1+sqrt(3)})++(120*\i-60:1) -- ++(120*\i-30:0.5);
		\draw[dashed] (120*\i+30:{1+sqrt(3)})++(120*\i+60:1) -- ++(120*\i+90:0.5);
		\draw[dashed] (120*\i+30:{1+sqrt(3)})++(120*\i+120:1) -- ++(120*\i+90:0.5);
	}	
	\end{scope}
  	\end{tikzpicture}
\end{subfigure}
\hfill
\begin{subfigure}{0.45\textwidth}
	\centering
	\begin{tikzpicture}
	\begin{scope}[scale=0.8]
	\foreach \i in {0,1,2} {
		\draw[dashed] (120*\i:1)--(120*\i+60:1);
		\draw (120*\i+60:1) -- (120*\i+120:1);
%		\node at (120*\i-30:1.1) {$a$};
%		\node at (120*\i+30:1.6) {$b$};
		\draw ++(120*\i+30:{0.7+sqrt(3)/2})+(120*\i+120:0.5) --+(120*\i-60:0.5);
		\draw[dashed] ++(120*\i+30:{0.7+sqrt(3)/2})+(120*\i+120:0.5) to[out=120*\i-120,in=120*\i+180] +(120*\i-60:0.5);
		\draw[dashed] ++(120*\i+30:{0.7+sqrt(3)/2})+(120*\i+120:0.5) to[out=120*\i,in=120*\i+60] +(120*\i-60:0.5);
		\draw[dashed] ++(120*\i+30:{1.4+sqrt(3)/2})+(120*\i+120:0.5) --+(120*\i-60:0.5);
		\draw[dashed] ++(120*\i+30:{2.4+sqrt(3)/2})+(120*\i+120:0.5) --+(120*\i-60:0.5);
		\draw ++(120*\i+30:{1.4+sqrt(3)/2})++(120*\i+120:0.5) --+(120*\i+30:1);
		\draw ++(120*\i+30:{1.4+sqrt(3)/2})++(120*\i-60:0.5) --+(120*\i+30:1);
%		\node at ($(120*\i+30:{(1+sqrt(3))/2})+(120*\i+120:0.7)$) {$c$};
%		\node at ($(120*\i+30:{(1+sqrt(3))/2})+(120*\i-60:0.7)$) {$c$};
		\draw (-90:{3.8+sqrt(3)})++(120*\i:1) -- ++(120*\i+120:1);
		\draw[dashed] (-90:{3.8+sqrt(3)})++(120*\i:1) -- ++(120*\i-120:1);
	}
	\draw ++(-90:{3.1+sqrt(3)/2})+(0:0.5) --+(180:0.5);
	\draw[dashed] ++(-90:{3.1+sqrt(3)/2})+(0:0.5) to[out=120,in=60] +(180:0.5);
	\draw[dashed] ++(-90:{3.1+sqrt(3)/2})+(0:0.5) to[out=-120,in=-60] +(180:0.5);
	\node at (-30:1.3){$A$};
	\node at ($(30:{0.7+sqrt(3)/2})+(-60:0.9)$){$B$};
	\node at ($(30:{1.9+sqrt(3)/2})+(-60:0.9)$){$C$};
	\node at ($(150:{0.7+sqrt(3)/2})+(-120:0.9)$){$B$};
	\node at ($(150:{1.9+sqrt(3)/2})+(-120:0.9)$){$C$};
	\node at ($(-90:{0.7+sqrt(3)/2})+(0:0.9)$){$B$};
	\node at ($(-90:{1.9+sqrt(3)/2})+(0:0.9)$){$C$};
	\node at ($(-90:{3.1+sqrt(3)/2})+(0:0.9)$){$B$};
	\node at ($(-90:{3.8+sqrt(3)})+(30:1.3)$){$A$};
	\end{scope}
  	\end{tikzpicture}
\end{subfigure}
  	\caption{The labelled graph $X$ (left) and its 3-block-decomposition (right).}
  	\label{fig:caleygraphgroupac}
\end{figure}
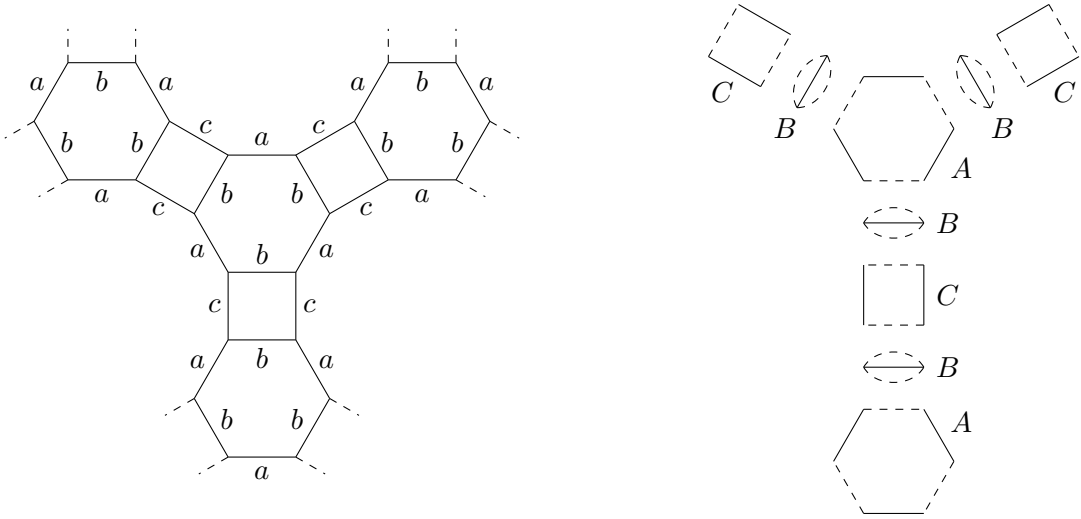

A grammar generating the languages of self-avoiding walks constructed as in the proof of Lemma \ref{lem:2conlangcf} is given by the following productions:
\begin{align*}
S\;\; &\vdash \; \mathrlap{\epsilon \mid V_{BA}^0 \mid V_{BC}^0 \mid b \mid b V_{BA}^1 \mid b V_{BC}^1 \mid U_{BA} V_{BC}^1 \mid U_{BC} V_{BA}^1\,,}\\[2pt]
 V_{BA}^0\; &\vdash \; \mathrlap{a \mid a V_{AB}^0 \mid a U_{AB} a \mid a U_{AB} a V_{AB}^0 \mid a U_{AB} a U_{AB} a\,,}\\[2pt]
 V_{BA}^1\; &\vdash \; \mathrlap{a \mid a V_{AB}^0 \mid a U_{AB} a \mid a U_{AB} a V_{AB}^0\,,}\\[2pt]
 V_{AB}^0\; &\vdash \; V_{BC}^0 \mid b \mid b V_{BC}^1\,,  & V_{AB}^1\; &\vdash \; V_{BC}^1\,, \\[2pt]
 V_{BC}^0\; &\vdash \; c \mid c U_{CB} c \mid c V_{CB}^0\,, & V_{BC}^1\; &\vdash \; c \mid c V_{CB}^0\,, \\[2pt]
 V_{CB}^0\; &\vdash \; V_{BA}^0 \mid b \mid b V_{BA}^1\,, & V_{CB}^1\; &\vdash \; V_{BA}^1\,, \\[2pt]
 U_{AB}\; &\vdash \; b \mid U_{BC}\,,  & U_{BA}\; &\vdash \; a U_{AB} a U_{AB} a\,, \\[2pt]
 U_{CB}\; &\vdash \; b \mid U_{BA}\,,  & U_{BC}\; &\vdash \; c U_{CB} c\,. 
\end{align*}

Translating this set of productions into the corresponding system of equations and solving this system yields the SAW-generating function
\begin{equation*}
	F_{SAW}(t|o)=\frac{P(t)+Q(t) \sqrt{-4 t^8-4 t^6+1}}{t^{12} \left(2 t^{10}+8 t^9+13 t^8+12 t^7+7 t^6+4 t^5+5 t^4+4 t^3+t^2-2 t-1\right)} ,
\end{equation*}
where $P(t)$ and $Q(t)$ are two polynomials of degree 23 and 17, respectively. The connective constant of $X$ is the reciprocal of the smallest positive root of the denominator of $F_{SAW}(t|o)$:
\begin{equation*}
\mu(X) \approx 1.8306977.
\end{equation*}
\end{exa}

\end{document}